\documentclass[11pt,a4paper,leqno]{amsart}

\usepackage{amssymb}
\usepackage{mathrsfs}
\usepackage[centertags]{amsmath}
\usepackage{amsfonts}
\usepackage{amsthm}
\usepackage{color}

\usepackage{fancyhdr}
\usepackage{amssymb}

\usepackage[normalem]{ulem} 
\usepackage{soul} 

\newcommand{\Leb}{\mathrm{Leb}}

\def\loc{\mathrm{loc}}

\theoremstyle{plain}

\newcommand\n[1]{\Vert #1 \Vert}

\newtheorem{theorem}{Theorem}[section]

\newtheorem{proposition}[theorem]{Proposition}
\newtheorem{assumption}[theorem]{Assumption}
\newtheorem{defi}[theorem]{Definition}
\newtheorem{lemma}[theorem]{Lemma}

\newtheorem{remark}[theorem]{Remark}
\newtheorem{acknowledgements}{Acknowledgements}

\newtheorem{coro}{Corollary}[section]

\input colordvi
\begin{document}

\baselineskip 16pt
\numberwithin{equation}{section}
\title[Maximal Inequalities for stochastic convolutions]
{Maximal inequalities for Stochastic convolutions driven by compensated
Poisson random measures in Banach spaces}

\author[Jiahui Zhu, Z. Brze\'{z}niak and E. Hauseblas]{{\bf Jiahui Zhu$\sharp\ddag$,  Zdzis{\l}aw Brze\'{z}niak$^\star$ and
 Erika Hausenblas$^\dag$
} \\
\\  $\hspace{10cm}$
\\
{\scriptsize{$^\star$ Department of Mathematics, University of York, Heslington, York, YO10 5DD, UK}\\
 \scriptsize{$^\dag$ Montanuniversity Leoben
Franz Josefstr 18,
8700 Leoben, Austria}\\
\scriptsize{$^\sharp$ College of Science, Zhejiang University of Technology, 310014 Hangzhou, China}\\
\scriptsize{$^\ddag$ The corresponding author.} }
}

\date{\today}
\begin{abstract}
Let $(E, \| \cdot\|)$ be a Banach space such that, for some $q\geq 2$, the  function $x\mapsto \|x\|^q$ is of $C^2$ class and its first and second Fr\'{e}chet derivatives are bounded by some constant multiples of $(q-1)$-th power of the norm and $(q-2)$-th power of the norm and let $S$ be a $C_0$-semigroup of contraction type  on $(E, \| \cdot\|)$.
We consider the following stochastic
convolution process
      \begin{align*}
   u(t)=\int_0^t\int_ZS(t-s)\xi(s,z)\,\tilde{N}(\mathrm{d} s,\mathrm{d} z), \;\;\; t\geq 0,
\end{align*}
where
$\tilde{N}$ is a compensated Poisson random measure on a   measurable space $(Z,\mathcal{Z})$  and $\xi:[0,\infty)\times\Omega\times Z\rightarrow E$
is an  $\mathbb{F}\otimes \mathcal{Z}$-predictable function.
We prove that there exists   a c\`{a}dl\`{a}g  modification a $\tilde{u}$ of the process $u$ which satisfies  the following
maximal inequality
\begin{align*}
      \mathbb{E} \sup_{0\leq s\leq t} \|\tilde{u}(s)\|^{q^\prime}\leq C\ \mathbb{E} \left(\int_0^t\int_Z \|\xi(s,z) \|^{p}\,N(\mathrm{d} s,\mathrm{d} z)\right)^{\frac{q^\prime}{p}},
\end{align*}
for all $ q^\prime \geq q$ and  $1<p\leq 2$ with $C=C(q,p)$.
\end{abstract}


\maketitle

\section{Introduction}
\label{sec-intro} Maximal inequalities
 for stochastic convolutions in the setting of Hilbert spaces or finite dimensional spaces have received considerable attention for many years. Ichikawa \cite{[Ichikawa]} considered maximal inequalities for
 $C_0$-semigroups of  contractions and right continuous martingales in
Hilbert spaces,  see also
Tubaro \cite{[Tubaro]}.
  A submartingale type inequality for stochastic convolutions of $C_0$-semigroups of contractions and
 square integrable martingales, also  in Hilbert spaces, was obtained
 by Kotelenez \cite{[Kotelenez]}. Kotelenez proved
 the existence of a c\`{a}dl\`{a}g version of the stochastic convolution process for square integrable
  c\`{a}dl\`{a}g martingales.  In a paper by the second named author and Peszat \cite{[Brz+Pesz_2000]}, the authors
  established  a maximal inequality in a certain class of Banach spaces for the
  stochastic convolution process driven by a Wiener process. Recently, this maximal inequality was generalized by van Neerven and the first named author to $C_0$-contraction semigroups on 2-smooth Banach spaces. Since many results obtained in the Wiener case may fail in pure jump type models, maximal inequalities for compensated Poisson random measures deserve an independent investigation. Here we extend the results from \cite{[Brz+Pesz_2000]} to the case
  where the stochastic convolution is driven by a compensated Poisson random measure. We work in the framework of stochastic integrals and convolutions driven by compensated Poisson random measures recently introduced by the second and third authors in \cite{[Brz+Haus_2009]}.

 Let us now briefly present the content of the paper.
 In the first section (i.e. section \ref{sec-si}) we set up notations and terminologies and then summarize without proofs some of the standard facts on stochastic integrals with values in martingale type $p$, $p\in (1,2]$, Banach spaces, driven by a  compensated Poisson random measure $\tilde{N}$ on a   measurable space $(Z,\mathcal{Z})$. Section \ref{sec-sc} is devoted to the study of the stochastic convolution process $(u(t))_{t \geq 0}$ driven by  $\tilde{N}$ defined by the following formula
 \begin{align}\label{SC}
\begin{split}
    u(t)&=\int_0^t \int_Z S(t-s)\xi(s,z)\,\tilde{N}(\mathrm{d} s,\mathrm{d} z),\ \ t \geq 0,\\
    \end{split}
\end{align}
where $S(t)$, $t\geq 0$ is a $C_0$-contraction semigroup (with the infinitesimal generator $A$) on a
martingale type $p$, $p\in (1,2]$, Banach space $E$  and $\xi:[0,\infty)\times\Omega\times Z\rightarrow E$
is an  $\mathbb{F}\otimes \mathcal{Z}$-predictable function such that for all $T>0$, $
     \int_0^T\int_Z\mathbb{E} |f(t,z)|^p_{E}\,\nu(\mathrm{d} z)\,\mathrm{d} t<\infty$. 
 In particular, we show that there exists a predictable version of the stochastic convolution process $u$. Under some suitable assumptions  we show that the process  $u$ is a unique strong solution to the following stochastic evolution equation
 \begin{align}\label{SDE}
\begin{split}
    du(t)&=Au(t)\,\mathrm{d} t+\int_Z\xi(t,z)\,\tilde{N}(\mathrm{d} t,\mathrm{d} z),\ \ t \geq 0,\\
    u(0)&=0.
    \end{split}
\end{align}
 In  section \ref{sec-main}  we present our main results. In particular,  the maximal inequalities are stated and proved when the $q$-th power, for some $q\geq p$, of some equivalent  norm on $E$ is of  $C^2$ class and its first and second Fr\'{e}chet derivatives are bounded by some constant multiples of $(q-1)$-th power of the norm and $(q-2)$-th power of the norm and $S(t)$, $t\geq 0$ is a $C_0$-contraction semigroup on $E$ with respect to this equivalent norm.  For the readers convenience let us state this result.

\begin{theorem}\label{theorem-main}
Suppose that $E$ is a real separable Banach space satisfying
Assumption \ref{assu-01} and   $S(t)$, $t\geq 0$ is a $C_0$ semigroup on $E$
satisfying
Assumption \ref{assu-02}.
 In the above described framework,  there exists a separable and c\`{a}dl\`{a}g
modification $\tilde u$ of the process $u$ defined by formula \eqref{sc-01}. Moreover, for every    $q^\prime\geq q$, where $p$ and  $q$ are  the numbers from Assumption \ref{assu-01}, there exists a  constant $C$ independent of the process $\xi$, such that for
every  stopping time $\tau>0$ and every $t>0$,
\begin{align}\label{ineq-main_intro}
      \mathbb{E}\sup_{0\leq s\leq t\wedge\tau}|\tilde u(s)|_E^{q^\prime}\leq C\ \mathbb{E}\left(\int_0^{t\wedge\tau}\int_Z|\xi(s,z)|_E^{p}\,N(\mathrm{d} s,\mathrm{d} z)\right)^{\frac{q^\prime}{p}}, \; t \geq 0.
\end{align}
\end{theorem}

Note that under above assumption on $E$, it can be shown, see Appendix \ref{sec-appendix},  that the Banach space $E$ satisfying the above condition is of martingale type $p$, for all $p\in(1,2]$. Hence both the stochastic integral and the stochastic convolution process are well defined.

Let us point out an important consequence of the above result. Namely,  if the right-hand side above is finite, then the stochastic convolution process $u$ admits an c\`adl\`ag modification,  see \cite{[Brz+5]} and a positive result given in \cite{[Liu+Zhai_2012]} for somehow related results in the Hilbert space framework.

 In the last part of section \ref{sec-main} we formulate and prove a different version of the maximal inequality. To be more precise, if $p \in [\sqrt{2},2]$ and $n\geq [\frac{\ln q}{\ln p}]$, then
for
every  stopping time $\tau>0$,
\begin{align}\label{inequality_13-intro}
      \mathbb{E}\sup_{0\leq s\leq t \wedge \tau }|\tilde{u}(s)|_E^{p^n}\leq C\
     \sum_{k=1}^n \mathbb{E}\left(\int_0^{t \wedge \tau}\int_Z |\xi(s,z)|_E^{p^k}\,\nu (\mathrm{d} z) \mathrm{d} s\right)^{{p^{n-k}}},\ t \geq 0.
\end{align}

In a brief section \ref{sec-pm}, we present extensions of the previous result to progressively measurable integrands.

\begin{remark}\label{rem-Nagy}
It is  possible to prove inequality \eqref{SC} by the method based on the Szek{\"{o}}falvi-Nagy Theorem on unitary dilations,  used earlier in \cite{[Haus+Seidler]}, see inequality (4) therein. The latter result has recently been generalized  to Banach spaces of finite cotype by Fr\"ohlich and Weis \cite{[Fr+Weis]}. However, this method  works only for analytic semigroups of contraction type. The results from the current paper are valid for all  $C_0$-semigroups of contraction type.
To be more precise, if  $E$ is a Banach space of martingale type $p$ , with $1<p\leq 2$, and  $A$ generates an analytic semigroup of contraction type on $E$, then,  by following almost  the  same lines as  in \cite{[Haus+Seidler]},  one can prove that for every $T>0$  there exists a constant $C>0$ such that for all progressively measurable processes $\xi$
\begin{align*}
      \mathbb{E}\sup_{0\leq s\leq t}|\tilde{u}(s)|_E^{q^\prime}\leq C\ \mathbb{E}\left(\int_0^t\int_Z|\xi(s,z)|_E^{p}\,N(\mathrm{d} s,\mathrm{d} z)\right)^{\frac{q^\prime}{p}},\; t \in [ 0,T].
\end{align*}
\end{remark}
Let us finish this Introduction by commenting that the results presented are applicable to nonlinear SPDEs, e.g. stochastic Euler Equations. In the case of similar problems with the Gaussian noise, the paper \cite{[Brz+Pesz_2000]} on which to a large extent our current research is based on, was in some sense a byproduct of a previous study by the same authors for stochastic Euler Equations in \cite{[Brz+Pesz_2001]}. It turns out that applications to stochastic Navier-Stokes  Equations of our paper even before it's publication have been found in a recent paper by Fernando et al. \cite{Fernando+Sritharan_2010}. For related results for stochastic reaction diffusion equations obtained by different approach one can consult a paper \cite{marinelli} by Marinelli and R\"ockner.

 \section{Stochastic integral}\label{sec-si}

Let
$(\Omega,\mathcal{F},\mathbb{F},\mathbb{P})$, where $\mathbb{F}=(\mathcal{F}_t)_{t\geq0}$,
be a filtered probability space satisfying the usual hypothesis. Let
$(S,\mathcal{S})$ be a measurable space. We write
$\mathbb{N}$ for the set of all natural numbers and set
$\bar{\mathbb{N}}=\mathbb{N}\cup\{\infty\}$. We denote by
$\mathbb{M}_{\bar{\mathbb{N}}}(S)$ the space of all
$\bar{\mathbb{N}}$-valued measures on $(S,\mathcal{S})$ and
$\mathcal{B}(\mathbb{M}_{\bar{\mathbb{N}}}(S))$ the smallest
$\sigma$-field on $\mathbb{M}_{\bar{\mathbb{N}}}(S)$ with respect to
which all the mapping
$i_B:\mathbb{M}_{\bar{\mathbb{N}}}(S)\ni\mu\mapsto\mu(B)\in\bar{\mathbb{N}}$,
$B\in\mathcal{S}$, are measurable.
\begin{defi}\label{defi: poisson random measure}
  A Poisson random measure  on $(S,\mathcal{S})$ over $(\Omega,\mathcal{F},\mathbb{F},\mathbb{P})$ is a map $N:\Omega\rightarrow \mathbb{M}_{\bar{\mathbb{N}}}(S)$
   such that the  family $\{N(B): B\in \mathcal{S}\}$  of $\mathcal{F}$-measurable $\bar{\mathbb{N}}$-valued functions defined by $N(B):=i_B\circ N:\Omega\rightarrow\bar{\mathbb{N}}$  satisfies  the following conditions
  \begin{trivlist}
  \item[(1)]for any $B\in\mathcal{S}$ such that $\eta(B):=\mathbb{E}(N(B))<\infty$, $N(B)$ is a Poisson random variable with parameter $\eta(B)$, i.e.
            \begin{align*}
                 \mathbb{P}(N(B)=n)=e^{-\eta(B)}\frac{\eta(B)^n}{n!},\ \ \ \ n=0,1,2,\cdots;
            \end{align*}
  \item[(2)](independently scattered property) for any pairwise disjoint sets $B_1,\cdots,B_n\in\mathcal{S}$ such that $\eta(B_i)<\infty$, $i=1,\cdots,n$, the random variables
         $$N(B_1),\ \cdots,\ N(B_n)$$
         are independent.
  \end{trivlist}
\end{defi}

\begin{remark} In what follows we will often assume that $T\in(0,\infty)\cup\{\infty\}$. Then, if   $T=\infty$,  by $[0,T]$ we mean the half-line $[0,\infty)$ and $\mathcal{F}_{T}$ stands for $\mathcal{F}$. Similarly,  for $a<\infty$, $(a,\infty]$ (respectively $[a,\infty]$) stands for $(a,\infty)$ (respectively $[a,\infty)$).
\end{remark}

\begin{defi}\label{def-predictabilility} Let us fix $T\in(0,\infty)\cup\{\infty\}$.
Let $\mathcal{P}$ denote the $\sigma$-field on $[0,T]\times \Omega$ generated by all left-continuous and $\mathbb{F}$-adapted real valued processes. We call $\mathcal{P}$ the predictable $\sigma$-field.\\
Assume that $(Z,\mathcal{Z})$ is a measurable space. Let $\hat{\mathcal{P}}$ denote the $\sigma$-field on
$[0,T]\times\Omega\times Z$ generated by all functions
$g:[0,T]\times\Omega\times Z\rightarrow \mathbb{R}$ satisfying the following
properties
\begin{enumerate}
    \item[(1)] for every $t \in [0,T]$, the mapping $(\omega,z)\mapsto g(t,\omega,z)$ is $\mathcal{F}_t\otimes\mathcal{Z}/\mathcal{B}(\mathbb{R})$-measurable,
    \item[(2)] for every $(\omega,z)$, the path $t\mapsto g(t,\omega,z)$ is left-continuous.
\end{enumerate}
We say that an $E$-valued process $g:[0,T]\times\Omega\rightarrow E$ is predictable if it is $\mathcal{P}/\mathcal{B}(E)$-measurable.\\
We say that a function $f:[0,T]\times\Omega\times Z\rightarrow E$ is
$\mathbb{F}\otimes \mathcal{Z}$-predictable if it is
$\hat{\mathcal{P}}/\mathcal{B}(E)$-measurable.
\end{defi}

\begin{remark}\label{rem-predictable sigma field} The predictable $\sigma$-field $\mathcal{P}$ is also generated by the family $\mathcal{R}$ (see for instance Th. 3.3 in \cite{[Metivier]}) defined by
	   \begin{align*}
	           \mathcal{R}=\{\{0\}\times F:F\in\mathcal{F}_0\}\cup\{(s,t]\times F: F\in\mathcal{F}_s,0\leq s<t, t \in [0, T]\}.
	\end{align*}
	The sets belonging to the family $\mathcal{R}$ are usually called predictable rectangles.
   Similarly, one can show, see \cite{Zhu_2010_PhD thesis},  that the $\mathbb{F}\otimes \mathcal{Z}$-predictable $\sigma$-field $\hat{P}$ is generated by a family $\hat{R}$
	  \begin{align*}
	           \hat{\mathcal{R}}=\{\{0\}\times F\times B:F\in\mathcal{F}_0,B\in\mathcal{Z}\}\cup\{(s,t]\times F\times B: F\in\mathcal{F}_s,B\in\mathcal{Z},0\leq s<t\leq T\}.
	   \end{align*}
Note that a function $f:[0,T]\times\Omega\times Z\rightarrow E$ which is now called
$\mathbb{F}\otimes \mathcal{Z}$-predictable,  in \cite{Zhu_2010_PhD thesis} was called $\mathbb{F}$-predictable. We believe that our current terminology is more natural.
\end{remark}

Suppose that $(Z,\mathcal{Z})$ is a measurable space and $\nu$ is a non-negative $\sigma$-finite measure on it. Let $\Leb$ be the Lebesgue measure on $(\mathbb{R}_+,\mathcal{B}(\mathbb{R}_+))$. According to \cite{[Sato]}, there exists a Poisson random measure $N$ on $(\mathbb{R}_+\times Z,\mathcal{B}(\mathbb{R}_+)\otimes\mathcal{Z})$ with the parameter $\eta(B)=\mathbb{E} N(B)=\Leb\otimes\nu(B)$, for $B\in\mathcal{B}(\mathbb{R}_+)\otimes\mathcal{Z}$. In particular, $\eta (I \times A) = \Leb (I) \nu (A)$, for $ I \in \mathcal{B} ({\mathbb{R}}_{+})$  and $A \in \mathcal{Z}$. Here as usual we shall employ the notation
\begin{align*}
    \tilde{N}=N-\Leb\otimes\nu
\end{align*}
to denote the compensated Poisson random measure of $N$.

For $T\in (0,\infty)\cup \{\infty\}$, let ${\mathcal{M}^p([0,T]\times Z; \hat{\mathcal{P}};E)}$ denote the
linear space consisting of (equivalence classes of) all $\mathbb{F}\otimes \mathcal{Z}$-predictable functions
$f:[0,T]\times\Omega\times Z\rightarrow E$ such that
\begin{align}\label{integrability}
     \int_0^T\int_Z\mathbb{E} |f(t,z)|^p_{E}\,\nu(\mathrm{d} z)\,\mathrm{d} t<\infty.
\end{align}
In other words, ${\mathcal{M}^p([0,T]\times Z; \hat{\mathcal{P}};E)}$
is the usual $L^p$ space of $E$-valued functions on $[0,T]\times\Omega\times Z$ with respect to the $\sigma$-field $\hat{\mathcal{P}}$ and the measure
$\Leb\otimes\mathbb{P}\otimes\nu$.  By $\mathcal{M}^p_{\loc}([0,\infty)\times Z; \hat{\mathcal{P}};E)$ we denote a
linear space consisting of all $\mathbb{F}\otimes \mathcal{Z}$-predictable functions
$f:[0,\infty)\times\Omega\times Z\rightarrow E$ such that condition \eqref{integrability} is satisfied for all $T>0$.

Till the end of  this section, we will  briefly sketch how one  constructs,  the
integral
\begin{align*}
         \int_0^T\int_Zf(t,z)\,\tilde{N}(\mathrm{d} t,\mathrm{d} z), \mbox{ for every function }f\in {\mathcal{M}^p([0,T]\times Z; \hat{\mathcal{P}};E)}.
\end{align*}
This integral we shall call the stochastic integral with respect to the
compensated Poisson random measure $\tilde{N}$. Full details of the definition can be found in \cite{Zhu_2010_PhD thesis}.

\begin{defi}\label{def-step-function}A function $f:[0,T]\times\Omega\times Z\rightarrow E$ is called a step function if there exists a finite sequence of numbers $0=t_0< t_1<\cdots<t_n=T$ and a finite family $A_{j-1}^k$, $j=1,\cdots,n$, $k=1,\cdots,m$, of sets from
$\mathcal{Z}$ with $\nu(A_{j-1}^k)<\infty$  such that
   \begin{align}\label{step function}
       f(t,\omega,z)=\sum_{k=1}^{m}\sum_{j=1}^{n}\xi^{k}_{j-1}(\omega)1_{(t_{j-1},t_j]}(t)1_{A^k_{j-1}}(z), \ \ (t,\omega,z)\in[0,T]\times\Omega\times Z,
   \end{align}
   where $\xi^k_{j-1}$ is an $E$-valued $\mathcal{F}_{t_{j-1}}$-measurable random variable, for every $j=1,\cdots,n$ and $k=1,\cdots,m$, and for each $j=1,\cdots,n$, the sets $A_{j-1}^k$, $k=1,\cdots,m,$ are pairwise disjoint.
\end{defi}
Note that each step function as defined in Definition \ref{def-step-function} is $\hat{\mathcal{P}}/\mathcal{B}(E)$-measurable. In other words, every step function is $\mathbb{F}\otimes \mathcal{Z}$-predictable.
   The class of all step functions satisfying \eqref{integrability} will be
denoted by
$\mathcal{M}^p_{step}([0,T]\times Z;\hat{\mathcal{P}};E)$.

\begin{defi} Assume that $f$ is a step function in $\mathcal{M}^p_{step}([0,T]\times Z;\hat{\mathcal{P}};E)$ of the form \eqref{step function}. If $t \in [0,T]$, then the  stochastic integral over the interval $[0,t]$ of a step function $f$ in $\mathcal{M}^p_{step}([0,T]\times Z;\hat{\mathcal{P}};E)$ of the form \eqref{step function} with respect to $\tilde{N}$  is a random variable $I_t(f)$, defined by
  \begin{align*}
      I_t(f):=\sum_{k=1}^{m}\sum_{j=1}^{n}\xi^k_{j-1}(\omega)\tilde{N}((t_{j-1}\wedge t,t_j\wedge t]\times A_{j-1}^k).
  \end{align*}
\end{defi}
Note that, for every $f \in
\mathcal{M}^p_{step}([0,T]\times Z;\hat{\mathcal{P}};E)$,
$I_t(f)$   is
linear with respect to $f$ and satisfies the following inequality, see  Lemma C.2 in \cite{[Brz+Haus_2009]} and, for related results, a recent paper \cite{Dirksen_2014} by Dirksen,
\begin{align}\label{eq-11}
            \mathbb{E} \left|I_t(f)\right|_E^p\leq C\,\mathbb{E} \int_0^t\int_Z|f(s,z)|^p_E\,\nu(\mathrm{d} z)\,\mathrm{d} s,
\end{align}
where $C$ is the same
constant as the one in the martingale-type-$p$ property of the space
$E$. Moreover, the process $I_t(f)$, $t \in[0,T]$
is an $E$-valued, mean $0$ and  c\`{a}dl\`{a}g  $\mathbb{F}$-martingale.

 Let us describe now how this definition can be extended to all functions in ${\mathcal{M}^p([0,T]\times Z; \hat{\mathcal{P}};E)}$.
Take $f\in{\mathcal{M}^p([0,T]\times Z; \hat{\mathcal{P}};E)}$. Then we
can find a sequence
$\{f^n\}_{n=1}^\infty$ of functions in $\mathcal{M}^p_{step}([0,T]\times Z;\hat{\mathcal{P}};E)$, see Theorem 3.2.23 in \cite{Zhu_2010_PhD thesis},
such that
\begin{align}\label{sec-2-eq-21}
         \mathbb{E} \int_0^{T}\int_Z|f(t,\omega,z)-f^n(t,\omega,z)|^p_E\,\nu(\mathrm{d} z)\,\mathrm{d} t\rightarrow 0,\ \ \text{as}\
         n\rightarrow\infty.
\end{align}
It follows from \eqref{eq-11} that
\begin{align*}
         \mathbb{E} \left|I_T(f^n)-I_T(f^m)\right|_E^p\leq C\mathbb{E} \int_0^T\int_Z|f^n(s,z)-f^m(s,z)|^p_E\,\nu(\mathrm{d} z)\,\mathrm{d} s\rightarrow0,
\end{align*}
as $n,m\rightarrow\infty$. In other words, $\{
I_T(f^n)\}_{n=1}^{\infty}$ is a Cauchy sequence in
$L^p(\Omega,E,\mathcal{F}_T)$. Thus this sequence
$\{
I_T(f^n)\}_{n=1}^{\infty}$
of random variables will converge in
$L^p(\Omega,\mathcal{F}_T;E)$ to some particular random variable
which we shall denote by $I_T(f)$ or $\int_0^T\int_Zf(s,z)\tilde{N}(ds,dz)$. Moreover, it does not depend on the choice of the sequence $\{f^n\}_{n=1}^\infty$ of
approximating step functions. We usually call $I_T(f)$ the
stochastic integral of $f$ with respect to the compensated Poisson
random measure $\tilde{N}$. For $0\leq a\leq b\leq T$,
$B\in\mathcal{Z}$ and $f\in{\mathcal{M}^p([0,T]\times Z; \hat{\mathcal{P}};E)}$, since
$1_{(a,b]}1_Bf$ is also in ${\mathcal{M}^p([0,T]\times Z; \hat{\mathcal{P}};E)}$, we can
define the stochastic integral from $a$ to $b$ of the function
$f\in{\mathcal{M}^p([0,T]\times Z; \hat{\mathcal{P}};E)}$ by
\begin{align*}
      I_{a,b}^B(f)=\int_a^b\int_Bf(s,z)\,\tilde{N}(\mathrm{d} s,\mathrm{d} z)=I_T(1_{(a,b]}1_Bf).
\end{align*}
For simplicity, we denote
     $ I_{t}(f):=I_T(1_{(0,t]}f)$, for $t\in(0,T]$ and set $I_t(f)=0$, when $t=0$.
Let $f\in{\mathcal{M}^p([0,T]\times Z; \hat{\mathcal{P}};E)}$. It was shown in \cite{[Brz+Haus_2009]} and \cite{Zhu_2010_PhD thesis} (see also \cite{Rudiger_2004} for the case $p=2$) that
the process $I_t(f)$, $t \in[0,T]$ is a c\`{a}dl\`{a}g
$p$-integrable $\mathbb{F}$-martingale with mean $0$. In particular, $I_t(f)$
has a modification which has $\mathbb{P}$-a.s.\ c\`{a}dl\`{a}g
trajectories and satisfies the following inequality
\begin{align}\label{Ito-isometry}
    \mathbb{E} |I_t(f)|^p_E=\mathbb{E} \big|\int_0^t\int_Zf(s,z)\,\tilde{N}(\mathrm{d} s,\mathrm{d} z)\big|_E^p\leq C\,\mathbb{E} \int_0^t\int_Z|f(s,z)|^p_E\,\nu(\mathrm{d} z)\,\mathrm{d} s.
\end{align}
From now on, while considering the stochastic process
$\int_0^t\int_Zf(s,z)\tilde{N}(ds,dz)$, $t \in[0,T]$,
we will assume that it has $\mathbb{P}$-a.s. c\`{a}dl\`{a}g trajectories.

\begin{remark}\label{rem-predictable versus progressively measurable}
We have defined the stochastic integral for integrands belonging to the class \\${\mathcal{M}^p([0,T]\times Z; \hat{\mathcal{P}};E)}$ of predictable processes. One should note that in the paper  \cite{[Brz+Haus_2009]} by the second and third authors, the integral is defined for an analogous class of progressively measurable processes. See also \cite{Rudiger_2004}. This approach is also discussed in section \ref{sec-pm} of the present paper.
\end{remark}

            If $\tau$ is a stopping time with $\mathbb{P}\{\tau\leq T\}=1$, we may set, for $\omega\in\Omega$,
\begin{align}\label{sec-2-eq-24}
	I_{\tau}(f)(\omega)=I_t(f)(\omega)\ \ \text{with }t=\tau(\omega).
\end{align}
Analogously, we shall also use the notation $I_{\tau}(f)=:\int_0^{\tau}\int_Zf(s,z)\tilde{N}(ds,dz)$. In this case, one can show that
\begin{align}\label{sec-2-stopped-SI}
      \int_0^{\tau}\int_Zf(s,z)\,\tilde{N}(\mathrm{d} s,\mathrm{d} z)=\int_0^{T}\int_Z1_{(0,\tau]}(s)f(s,z)\,\tilde{N}(\mathrm{d} s,\mathrm{d} z),\ \mathbb{P}\text{-a.s.}.
\end{align}

\begin{remark}\label{rem-stopped integral}
Note that since $\tau$ is a stopping time (without any additional property), the random process \[1_{(0,\tau]}:[0,\infty)\times \omega \ni (s,\omega)\mapsto 1_{(0,\tau(\omega]}(s) \in [0,1] \] is predictable, see the comment after \cite[Definition IV.5.3]{Revuz+Yor_1999} or Proposition 4.6 in \cite{[Metivier]}. Hence, provided that the process $f$ belongs to ${\mathcal{M}^p([0,T]\times Z; \hat{\mathcal{P}};E)}$, the process $1_{(0,\tau]}f$ belongs to that space as well. In particular, the integral on the RHS of
\eqref{sec-2-stopped-SI} is well defined.
\end{remark}

Indeed, it  can be easily verified that for a step function $f^n$ of the form \eqref{step function} one has
\begin{align}\label{sec-2-eq-26}
	    \int_0^T\int_Z1_{(0,\tau]}(s)f^n(s,z)\,\tilde{N}(\mathrm{d} s,\mathrm{d} z)=\int_0^{\tau}\int_Zf^n(s,z)\,\tilde{N}(\mathrm{d} s,\mathrm{d} z),\ \mathbb{P}\text{-a.s.}
\end{align} Take $f\in {\mathcal{M}^p([0,T]\times Z; \hat{\mathcal{P}};E)} $. Then as we discussed before, there exists an $\mathcal{M}^p_{step}((0,T]\times Z;\hat{\mathcal{P}};E)$-valued sequence $\{f^n\}$ such that
\begin{align*}
         \mathbb{E} \int_0^{T}\int_Z1_{(0,\tau]}(s)|f(t,z)-f^n(s,z)|^p_E\,\nu(\mathrm{d} z)\mathrm{d} t\rightarrow 0,\ \ \text{as}\
         n\rightarrow\infty.
\end{align*}
By applying inequality \eqref{Ito-isometry}, we infer
\begin{align*}
	     \lim_{n\rightarrow\infty} \mathbb{E} \Big|\int_0^T\int_Z1_{(0,\tau]}(s)f^n(s,z)\,\tilde{N}(\mathrm{d} s,\mathrm{d} z)-\int_0^T\int_Z1_{(0,\tau]}(s)f(s,z)\,\tilde{N}(\mathrm{d} s,\mathrm{d} z)\Big|^p_E=0.
	\end{align*}
	  Hence, we can extract a subsequence (denoted again by the same notation for simplicity) such that
	 \begin{align}\label{sec-2-eq-22}
		\lim_{n\rightarrow\infty}\int_0^T\int_Z1_{(0,\tau]}(s)f^n(s,z)\,\tilde{N}(\mathrm{d} s,\mathrm{d} z)=\int_0^T\int_Z1_{(0,\tau]}(s)f(s,z)\,\tilde{N}(\mathrm{d} s,\mathrm{d} z),\ \mathbb{P}\text{-a.s.}
		\end{align}

   On the other hand, since $I_t(f)-I_t(f^n)$ is a right-continuous martingale,  $|I_t(f)-I_t(f^n)|_E$ is a real-valued submartingale. Since $\mathbb{P}(\tau\leq T)=1$, by the stopped Doob inequality, we have
	\begin{align*}
		 \mathbb{E} |I_{\tau}(f)-I_{\tau}(f^n)|_E^p\leq  (\frac{p}{p-1})^p\,\mathbb{E} |I_T(f)-I_T(f^n)|^p_E\rightarrow0,\ \text{as }n\rightarrow\infty .
	\end{align*}
It follows (by selecting a further subsequence) that
	\begin{align}\label{sec-2-eq-23}
		\lim_{n\rightarrow\infty}\int_0^{\tau}\int_Zf^n(s,z)\,\tilde{N}(\mathrm{d} s,\mathrm{d} z)=\int_0^{\tau}\int_Zf(s,z)\,\tilde{N}(\mathrm{d} s,\mathrm{d} z),\ \mathbb{P}\text{-a.s.}
		\end{align}
 Hence, \eqref{sec-2-stopped-SI} follows directly from \eqref{sec-2-eq-26}, \eqref{sec-2-eq-22} and \eqref{sec-2-eq-23}.	
	
In particular, for every $t\in[0,T]$ and any stopping time $\tau$ with $\mathbb{P}(\tau<\infty)=1$, we have
\begin{align}\label{lateron}
      \int_0^{t\wedge\tau}\int_Zf(s,z)\,\tilde{N}(\mathrm{d} s,\mathrm{d} z)=\int_0^{t}\int_Z1_{(0,\tau]}(s)\, f(s,z)\,\tilde{N}(\mathrm{d} s,\mathrm{d} z).
\end{align}

    We will also need the following result about the integral $\int_0^t\int_Z f(s,\omega,z) N(\mathrm{d} s,\mathrm{d} z)(\omega)$, which is defined, for every  $\omega\in\Omega$, as Bochner integral with respect to measure $N(\mathrm{d} s,\mathrm{d} z)(\omega)$ on $[0,t]\times Z$.

Later on we will use in the proof the notion of point processes (see \cite{[Ikeda]}). For more details we refer the reader to \cite{[Brz+Liu+Zhu]} or \cite{Zhu_2010_PhD thesis}.
Let us assume that $\pi$ is a stationary Poisson point process on $(Z,\mathcal{Z})$ with the intensity measure $\nu$, see  \cite[Theorem 54]{[Rong]} for the existence of such a  process.
For simplicity of notation, the Poisson random measure associated to the Poisson point process $\pi$ will be still denoted by $N$. We use the notation $\tilde{N}(t,A)=N(t,A)-t\nu(A)$, $t\geq0$, $A\in\mathcal{Z}$ to denote its compensated Poisson random measure.

\begin{proposition} \label{prop-sec-2}
   If $T\in (0,\infty)\cup \{\infty\}$ and $f:[0,T]\times\Omega\times Z\rightarrow E$ is a $\mathcal{B}([0,T])\otimes\mathcal{F}_T\otimes\mathcal{Z}$-measurable function and
\begin{align}\label{sec-2-prop-eq-1}
	 \mathbb{E} \int_0^T\int_Z|f(s,z)|_E\,N(\mathrm{d} s,\mathrm{d} z)<\infty,
\end{align} then we have for every $t\in [0,T]$,
\begin{align}\label{sec-2-eq-30}
	\int_0^t\int_Z f(s,\omega,z) N(\mathrm{d} s,\mathrm{d} z)(\omega)=\sum_{s\leq t}f(s,\omega,\pi(s,\omega)),\ \mathbb{P}-a.s.
\end{align}
\end{proposition}
\begin{proof}
	Since $f$ is $\mathcal{B}([0,T])\otimes\mathcal{F}_T\otimes\mathcal{Z}$-measurable, for every $\omega\in\Omega$, $f(\cdot,\omega,\cdot)$ is $\mathcal{B}([0,T])\otimes\mathcal{Z}$-measurable. By \eqref{sec-2-prop-eq-1}, we deduce that  $\int_0^T\int_Z|f(s,\omega,z)|_EN(\mathrm{d} s,\mathrm{d} z)<\infty$, $\mathbb{P}$-a.s. Let us choose for the remainder of the proof an $\omega\in\Omega$ such that the previous integral is finite. Hence, $f(\cdot,\omega,\cdot)$ is Bochner integrable with respect to $N(\mathrm{d} s,\mathrm{d} z)(\omega)$. Moreover, we can find a sequence $\{f^n\}$ of functions on $[0,T]\times Z$ of the form $\sum_{i=1}^mx_i1_{B_i}$, $x_i\in E$ and $B_i\in\mathcal{B}([0,T])\otimes\mathcal{Z}$ such that $|f^n(t,\omega,z)-f(t,\omega,z)|_E$ decreases to $0$ as $n\rightarrow\infty$, for all $(t,z)\in[0,T]\times Z$. Hence it is enough to show \eqref{sec-2-eq-30} for functions of the form  $\sum_{i=1}^ma_i1_{B_i}$. For this, observe that
	\begin{align*}
		 \int_0^t\int_Z f(s,\omega,z) N(\mathrm{d} s,\mathrm{d} z)(\omega)&=\sum_{i=1}^mx_iN(B_i)(\omega) =\sum_{i=1}^m x_i\sum_{s\in[0,t]\cap\mathcal{D}(\pi(\omega))}1_{B_i}(s,\pi(s,\omega))\\
		&= \sum_{s\in[0,t]\cap\mathcal{D}(\pi(\omega))}\sum_{i=1}^mx_i1_{B_i}(s,\pi(s,\omega))
		= \sum_{s\in[0,t]\cap\mathcal{D}(\pi(\omega))}f(s,\omega,\pi(s,\omega)).
	\end{align*}

\end{proof}

\section{Stochastic convolution}\label{sec-sc}

  In this section, we continue to assume that $E$ is a separable Banach space of martingale type $p$, where $p\in(1,2]$.
Let $(S(t))_{t\geq0}$ be a contraction $C_0$-semigroup on $E$ with the infinitesimal generator $A$. Let us denote by $R(\lambda,A)=(\lambda I-A)^{-1}$, $\lambda>0$, the resolvent operator of $A$ and by $A_{\lambda}=\lambda A(\lambda I-A)^{-1}$ the Yosida approximation of $A$. It is well known that $A_{\lambda}$ is a bounded operator on $E$ and  $A_{\lambda}x\rightarrow x$, as $\lambda\rightarrow\infty$, for $x\in\mathcal{A}$ and $\lambda R(\lambda,A)x\rightarrow x$, as $\lambda\rightarrow\infty$, for all $x\in E$. Moreover, $\lambda R(\lambda,A)x\in\mathcal{D}(A)$, for all $x\in E$.

\indent Suppose that $\xi\in{\mathcal{M}^p([0,T]\times Z; \hat{\mathcal{P}};E)} $ and $\tilde{N}$ is a compensated Poisson random measure corresponding to the point process $\pi=(\pi(t))_{t\geq0}$. The aim of this paper is to study the path properties of the stochastic convolution process $u$ defined by
\begin{align}\label{sto_convolution}
   u(t)=\int_0^t\int_ZS(t-s)\xi(s,z)\,\tilde{N}(\mathrm{d} s,\mathrm{d} z),\ \ t\in [0,T].
\end{align}

Now let us consider Problem \eqref{SDE}, which for the convenience of the
reader we rewrite below.
\begin{align}\label{SDE2}
\begin{split}
    \mathrm{d} u(t)&=Au(t)\,\mathrm{d} t+\int_Z\xi(t,z)\,\tilde{N}(\mathrm{d} t,\mathrm{d} z),\ \ t \in [0,T],\\
    u(0)&=0.
    \end{split}
\end{align}
In this chapter we will study the above Question under a stronger assumption on the process $\xi$, i.e. that $\xi\in \mathcal{M}^p([0,T]\times Z; \hat{\mathcal{P}};\mathcal{D}(A))$.

\begin{defi}\label{defi: strong solution} Suppose that $\xi\in \mathcal{M}^p([0,T]\times Z; \hat{\mathcal{P}};\mathcal{D}(A))$.
A strong solution to Problem \eqref{SDE2} on the time interval $[0,T]$ is a
$\mathcal{D}(A)$-valued $\mathbb{F}$-adapted stochastic process $(u(t))_{t \in  [0,T]}$ with $E$-valued c\`{a}dl\`{a}g trajectories
 such that
\begin{enumerate}
   \item[(1)]$u(0)=0$ a.s.
   \item[(2)]For any $t \in [0,T]$ the following equality holds $\mathbb{P}$-a.s.
   \begin{align}\label{strong solution}
    u(t)=\int_0^tAu(s)\,\mathrm{d} s+\int_0^{t}\int_Z\xi(s,z)\,\tilde{N}(\mathrm{d} s,\mathrm{d} z).
   \end{align}
  \end{enumerate}
Similarly we can define a strong solution to Problem \eqref{SDE2} if $T=\infty$ and $\xi\in\mathcal{M}^p_{\loc}([0,\infty)\times Z; \hat{\mathcal{P}};\mathcal{D}(A))$.
\end{defi}

\begin{lemma}\label{lemma:strong solution} Assume that  $\xi\in\mathcal{M}^p([0,T]\times Z; \hat{\mathcal{P}};\mathcal{D}(A))$ for some $T>0$.
Then the process $u$ defined by
\begin{align}\label{sec-2-mild-eq}
   u(t)=\int_0^t\int_ZS(t-s)\xi(s,z)\,\tilde{N}(\mathrm{d} s,\mathrm{d} z),\ t \geq 0,
\end{align}
is  a unique strong solution of Equation \eqref{SDE}. In particular, $u$ has  $E$-valued c\`{a}dl\`{a}g trajectories.

\end{lemma}
\begin{proof} Let us fix $T>0$ and $t\in [0,T]$.
Define a function $F:[0,t]\times E\ni(s,x)\mapsto S(t-s)x\in E$.
It is straightforward to see that the function $F$ is separately continuous. Moreover, since $[0,t]$ is compact, we infer that $F$ is continuous.
   This, together with the $\mathbb{F}\otimes \mathcal{Z}$-predictability assumption on $\xi$, implies that
 the composition mapping
   $$[0,t]\times\Omega\times Z\ni(s,\omega,z)\mapsto (s,\xi(s,\omega,z))\mapsto F(s,\xi(s,\omega,z))\in E
$$
   is $\mathbb{F}\otimes \mathcal{Z}$-predictable.
   On the other hand,  since each operator $S(t)$, $t\geq0$ is a contraction on $E$ and $\xi$ is in $\mathcal{M}^p([0,T]\times Z; \hat{\mathcal{P}};E) ,$ we have
   \begin{align*}
        \mathbb{E} \int_0^{T}|1_{(0,t]}(s)S(t-s)\xi(s,z)|^p_E\,\nu(\mathrm{d} z)\,\mathrm{d} s\leq\mathbb{E} \int_0^T|\xi(s,z)|^p_E\,\nu(\mathrm{d} z)\,\mathrm{d} s<\infty.
   \end{align*}
   Therefore, the process $S(t-s)\xi(s,z)$ belongs to
   $\mathcal{M}^p([0,T]\times Z; \hat{\mathcal{P}};E) $.
   Hence, since the number $t$ is fixed, the process defined by
\begin{align*}
       \int_0^r\int_ZS(t-s)\xi(s,z)\,\tilde{N}(\mathrm{d} s,\mathrm{d} z),\ \ r\in[0,t],
\end{align*}
is a $\mathbb{F}$-martingale on $[0,t]$, see \cite{Zhu_2010_PhD thesis}. In particular, for each $r\in[0,t]$, the random variable
$\int_0^r\int_ZS(t-s)\xi(s,z)\tilde{N}(ds,dz)$ is
$\mathcal{F}_r$-measurable and hence
$u(t)$ is
$\mathcal{F}_t$-measurable.

 Now we proceed to show that $u(t)\in\mathcal{D}(A)$.
 Since $\xi\in\mathcal{M}^p([0,T]\times Z; \hat{\mathcal{P}};\mathcal{D}(A))$ and $R(\lambda,A)A=\lambda R(\lambda,A)-I_E$ on $\mathcal{D}(A)$,
we obtain
\begin{align*}
     R(\lambda,A)\int_0^t\int_ZAS(t-s)\xi(s,z)\,\tilde{N}(\mathrm{d} s,\mathrm{d} z)
              &=\lambda
              R(\lambda,A)\int_0^t\int_ZS(t-s)\xi(s,z)\,\tilde{N}(\mathrm{d} s,\mathrm{d} z)\\
              &\hspace{1cm}-\int_0^t\int_ZS(t-s)\xi(s,z)\,\tilde{N}(\mathrm{d} s,\mathrm{d} z).
\end{align*}
Hence, it follows that
\begin{align*}
   \int_0^t\int_ZS(t-s)&\xi(s,z)\,\tilde{N}(\mathrm{d} s,\mathrm{d} z)\\
   &=R(\lambda,A)\left[\lambda\int_0^t\int_ZS(t-s)\xi(s,z)\,\tilde{N}(\mathrm{d} s,\mathrm{d} z)-\int_0^t\int_ZAS(t-s)\xi(s,z)\tilde{N}(\mathrm{d} s,\mathrm{d} z)\right].
\end{align*}
Since the range of $R(\lambda,A)$ is equal to $\mathcal{D}(A)$, we infer that
$u(t)\in \mathcal{D}(A)$.

Next we shall show that
\begin{align}\label{eq-14}
     A\int_0^t\int_ZS(t-s)\xi(s,z)\,\tilde{N}(\mathrm{d} s,\mathrm{d} z)=\int_0^t\int_ZAS(t-s)\xi(s,z)\,\tilde{N}(\mathrm{d} s,\mathrm{d} z),\ \ \mathbb{P}\text{-a.s.}
\end{align}
For this, let us take $h \in (0,t)$ and observe that since the operator $\frac{S(h)-I}{h}$ is bounded, we get the following equality
 \begin{align*}
\frac{S(h)-I}{h}\int_0^t\int_ZS(t-s)\xi(s,z)\,\tilde{N}(\mathrm{d} s,\mathrm{d} z)
=\int_0^t\int_Z\frac{S(h)-I}{h}S(t-s)\xi(s,z)\,\tilde{N}(\mathrm{d} s,\mathrm{d} z).
 \end{align*}
So by applying the triangle inequality and
\eqref{Ito-isometry}, we find
 \begin{eqnarray}\label{eq-13}
   &&\hspace{-1.5cm}\lefteqn{\mathbb{E} \left|A\int_0^t\int_ZS(t-s)\xi(s,z)\,\tilde{N}(\mathrm{d} s,\mathrm{d} z)-\int_0^t\int_ZAS(t-s)\xi(s,z)\,\tilde{N}(\mathrm{d} s,\mathrm{d} z)\right|^p_E\nonumber}\\
   &\leq& 2^p\,
   \mathbb{E} \left|A\int_0^t\int_ZS(t-s)\xi(s,z)\,\tilde{N}(\mathrm{d} s,\mathrm{d} z)-\frac{S(h)-I}{h}\int_0^t\int_ZS(t-s)\xi(s,z)\,\tilde{N}(\mathrm{d} s,\mathrm{d} z)\right|^p_E\nonumber\\
   &+& 2^p\,\mathbb{E} \left|\int_0^t\int_ZAS(t-s)\xi(s,z)\,\tilde{N}(\mathrm{d} s,\mathrm{d} z)-\int_0^t\int_Z\frac{S(h)-I}{h}S(t-s)\xi(s,z)\,\tilde{N}(\mathrm{d} s,\mathrm{d} z)\right|^p_E\nonumber\\
   &\leq&  2^p\,
   \mathbb{E} \left|\left(A-\frac{S(h)-I}{h}\right)u(t)\right|^p_E\nonumber\\
   &+&C_p\,\mathbb{E} \int_0^t\int_Z\left|AS(t-s)\xi(s,z)-\frac{1}{h}\Big{(}S(h)-I\Big{)}S(t-s)\xi(s,z)\right|^p_E\nu(\mathrm{d} z)\,\mathrm{d} s\nonumber\\
   &=:&\textrm{I}(h)+\textrm{II}(h).
 \end{eqnarray}
 First we will deal with $\textrm{II}(h)$. Since $\xi(s,z)\in \mathcal{D}(A)$, we observe that $$\frac{S(h)-I}{h}S(t-s)\xi(s,z)=\frac{1}{h}\int_0^h S(r)
 AS(t-s)\xi(s,z)\,\mathrm{d} r.$$ So by using the uniform condition of the operators $S(h)$ and $S(t-s)$, we deduce that  $$\Big|\frac{S(h)-I}{h}S(t-s)\xi(s,z)\Big|_E\leq
 C|A\xi(s,z)|_E.$$
Hence we infer that the integrand
 of $\textrm{II}(h)$ is bounded by a function $C_1|A\xi(s,z)|_E$
which belongs to $\mathcal{M}^p([0,T]\times
Z;\mathbb{R})$ by
assumption. Clearly, the integrand
$$\Big|AS(t-s)\xi(s,z)-\frac{1}{h}\Big{(}S(h)-I\Big{)}S(t-s)\xi(s,z)\Big|^p_E$$
converges to $0$ pointwise on $[0,t]\times\Omega\times Z$.
Therefore, by the Lebesgue
Dominated Convergence Theorem (LDCT for short), $\textrm{II}(h)$ converges to $0$ as $h\searrow0$.\\
Now we turn our attention to the term $\textrm{I}(h)$. Since $\mathbb{E} \|u(t)\|^p_{\mathcal{D}(A)}<\infty$ and $|\frac{1}{h}(S(h)-I)x|_E\leq |Ax|_E$, for all $x\in\mathcal{D}(A)$, we infer by applying the LDCT that  $\textrm{I}(h)$
converges to $0$ as $h\mathrm{d}ownarrow0$ as well. Hence \eqref{eq-14} holds.

 In order to finish the proof, we need to verify \eqref{strong solution}.
By \eqref{eq-14}  and the
Fubini Theorem we find
\begin{align*}
      \int_0^tAu(s)\,\mathrm{d} s&=\int_0^t\int_0^s\int_ZAS(s-r)\xi(r,z)\tilde{N}(\mathrm{d} r,\mathrm{d} z)\,\mathrm{d} s\\
      &=\int_0^t\int_Z\int_r^tAS(s-r)\xi(r,z)\,\mathrm{d} s\,\tilde{N}(\mathrm{d} r,\mathrm{d} z)\\
      &=\int_0^t\int_Z\left(S(t-r)\xi(r,z)-\xi(r,z)\right)\,\tilde{N}(\mathrm{d} r,\mathrm{d} z)\\
      &=u(t)-\int_0^t\int_Z\xi(r,z)\,\tilde{N}(\mathrm{d} r,\mathrm{d} z), \ \
      \mathbb{P}\text{-a.s.}
\end{align*}
Hence we have
\begin{align*}
    u(t)=\int_0^tAu(s)\,\mathrm{d} s+\int_0^{t}\int_Z\xi(s,z)\tilde{N}(\mathrm{d} s,\mathrm{d} z)
\end{align*}
from which we can also see that $u$ is a c\`{a}dl\`{a}g process.

\indent For the uniqueness, suppose that $u^1$ and $u^2$ are two
strong solutions of Problem \eqref{SDE}. Let $w=u^1-u^2$. Then we
infer
\begin{align*}
 w(t)=u^1(t)-u^2(t)=\int_0^tA(u^1(s)-u^2(s))\,\mathrm{d} s=A\int_0^tw(s)\,\mathrm{d} s.
\end{align*}
Put $v(t)=\int_0^tw(s)\,ds$. Then $v(t)$ is continuously
differentiable on $[0,T]$ and $v(t)\in\mathcal{D}(A)$. Now applying the
It\^{o} formula to the function $f(s)=S(t-s)v(s)$ yields
\begin{align*}
     \frac{\mathrm{d} f(s)}{\mathrm{d} s}&=-AS(t-s)v(s)+S(t-s)\frac{\mathrm{d} v(s)}{\mathrm{d} s}\\
                     &=-AS(t-s)v(s)+S(t-s)w(s)=-AS(t-s)v(s)+S(t-s)Av(s)=0.
\end{align*}
So we infer $v(t)=f(t)=f(0)=S(t)v(0)=0$ a.s.. Therefore, $w(t)=0$
a.s.. That is $u^1(t)=u^2(t)$ a.s. $t \geq 0$.
\end{proof}


\section{Maximal inequalities for stochastic convolution}\label{sec-main}

From now on we make the following assumptions on the Banach space $E$.
\begin{assumption}\label{assu-01}
Suppose that $\bigl(E,\vert \cdot \vert \bigr)$ is a real separable Banach space and $p \in (1,2]$. In addition we assume that the Banach space $E$
satisfies the following condition:\\
There exists a norm $|\cdot|_E$ on
$E$ which is equivalent to $\vert \cdot \vert$, and numbers $q\in [p,\infty)  $  such that
 the function $\phi:E\ni
x\mapsto |x|^q_E\in\mathbb{R}$, is of class $C^2$ and there exist
constants $k_1,k_2$ such that for every $x\in E$, the first the second Fr{`e}chet derivatives of $\phi$ satisfy,  respectively, $|\phi^\prime(x)|\leq
k_1|x|^{q-1}_E$ and $|\phi^{\prime\prime}(x)|\leq k_2|x|_E^{q-2}$, for all $x\in E$.
\end{assumption}
\begin{assumption}\label{assu-02}
 The $C_0$ semigroup $S(t)$, $t\geq 0$ on $E$ is  of contraction type with respect to the norm $|\cdot|_E$ from Assumption \ref{assu-01}.
\end{assumption}

Now we proceed with the study of the stochastic convolution
\begin{align}\label{sc-01}
   u(t)=\int_0^t\int_ZS(t-s)\xi(s,z)\tilde{N}(\mathrm{d} s,\mathrm{d} z),\; t\in [0,T],
\end{align}
provided $T>0$ and the process $\xi$ belongs to the space $ \mathcal{M}^p([0,T]\times Z; \hat{\mathcal{P}};E)$.

Let us now formulate our main result a proof of which will be presented at the end of this section and preceded by two Lemmata: \ref{Lemma 1 for theo} and \ref{lem:separability}.


\begin{theorem}\label{theo_1}
Suppose that $\bigl(E,\vert \cdot \vert \bigr)$ is a real separable Banach space satisfying
Assumption \ref{assu-01}, with numbers  $p\in (1,2]$ and  $q \in [p,\infty)$ and   $S(t)$, $t\geq 0$ is a $C_0$ semigroup on $E$
satisfying
Assumption \ref{assu-02}. Assume that $\xi\in \mathcal{M}^p_{\loc}([0,\infty)\times Z; \hat{\mathcal{P}};E)$ and in fact that  
\[
\mathbb{E} \left(\int_0^{T}\int_Z|\xi(s,z)|^{p}\,N(\mathrm{d} s,\mathrm{d} z)\right)^{\frac{q}{p}},\;\; T>0.
\] 
 Then there exists a separable and c\`{a}dl\`{a}g
modification $\tilde u$ of the process $u$ defined by formula \eqref{sc-01}. Moreover, for every    $q^\prime\geq q$,  there exists a  constant $C$ independent of the process $\xi$, such that for
every  stopping time $\tau>0$ and every $t>0$,
\begin{align}\label{inequality}
      \mathbb{E} \sup_{0 \leq s \leq t \wedge \tau }|\tilde{u}(s)|^{q^\prime}\leq C\ \mathbb{E} \left(\int_0^{t\wedge\tau}\int_Z|\xi(s,z)|^{p}\,N(\mathrm{d} s,\mathrm{d} z)\right)^{\frac{q^\prime}{p}}.
\end{align}
\end{theorem}

Before proceeding with proofs, let us point our an  important ingredient of the above result, i.e. the c\`adl\`g property of the stochastic convolution process $u$. This topic has attracted recently an attention  even in the Hilbert space setup because of the counterexample presented in \cite{[Brz+5]} and a positive result given in \cite{[Liu+Zhai_2012]}.

\begin{remark}\label{rem-4.6}
It is worth pointing out that since by  \cite[Proposition 3.6]{[Da prato]}, every adapted and stochastically continuous process on an interval $[0,T]$ has a predictable version on $[0,T]$, we conclude that the process
$u(t)$, $t \geq 0$ has a predictable version.
Henceforth, when we study the stochastic convolution process, we
refer to the version of it  that is c\`{a}dl\`{a}g and its
supremum over every compact interval  $[0,T]$ is $\mathcal{F}_T$-measurable.
\end{remark}

\begin{remark}\label{rem-cond 1} It can be proved, see Appendix \ref{sec-appendix}, that if the real separable Banach space $E$ satisfies  Assumption \eqref{assu-01}, then $E$ is of martingale type $p$, for all $p\in(1,2]$. Hence the stochastic convolution process \eqref{sc-01} is well defined for $\xi\in\mathcal{M}^p([0,T]\times Z; \hat{\mathcal{P}};E) $ and in particular, if $\xi\in\mathcal{M}^p([0,T]\times Z; \hat{\mathcal{P}};\mathcal{D}(A))$, then \eqref{sc-01} is a unique strong solution of equation \eqref{SDE} by Lemma \ref{lemma:strong solution}.
\end{remark}

\begin{remark}\label{rem-Sobolev spaces}
Note  that if $\mathcal{O}$ is a domain in $\mathbb{R}^n$ with Lipschitz boundary $\partial{\mathcal{O}}$, then the   Sobolev spaces $H^{s,r}(\mathcal{O})$ with
$r\in[2,\infty)$ and $s\in\mathbb{R}_+$ satisfy Assumption \eqref{assu-01} and  the Lebesgue  $L^{r}(\mathcal{O})$-spaces with $r\geq 2$ also satisfies Assumption \eqref{assu-01}.
\end{remark}

Before proving the main theorem, we first need the following Lemmas.
\begin{lemma}\label{Lemma 1 for theo} If Assumptions \ref{assu-01} and \ref{assu-02} are satisfied, then  for all $x\in D(A)$, \[\phi^\prime(x)(Ax)\leq 0.\]
\end{lemma}

\begin{proof}
This follows immediately from the fact that the function $t\mapsto\phi(S(t)x)$ is decreasing and
\begin{align*}
       \left.\frac{d\phi(S(t)x)}{dt}\right|_{t=0}=\phi^\prime(S(0)x)(Ax)=\phi^\prime(x)(Ax).
\end{align*}
\end{proof}

In the remainder of this section we will always assume that Assumptions \ref{assu-01} and \ref{assu-02} are satisfied.

\begin{lemma}\label{lem:separability} Suppose that  $\xi\in \mathcal{M}^p([0,T]\times Z; \hat{\mathcal{P}};E)$ for some $T>0$. Then there exists a version $\bar{u}$ of the process $u$ defined by equality \eqref{sc-01} such that the function $\sup_{t \in [0,T]}|\bar{u}(t)|$ is $\mathcal{F}_T$-measurable.
\end{lemma}

\begin{proof}
	According to Remark \ref{rem-cond 1}, we can take $p\in(1,2]$. Let us fix $T>0$ and that $\xi\in \mathcal{M}^p([0,T]\times Z; \hat{\mathcal{P}};E)$.
We begin with showing that the process $u$ is continuous in the $p$-mean.
	By applying the inequality $|a+b|^p\leq 2^p|a|^p+2^p|b|^p$,
	 inequality \eqref{Ito-isometry} and the contraction property of the
	semigroup $S(t)$, $t\geq0$, we have,  for $0\leq r<t\leq T$,
	 \begin{align*}
	        \mathbb{E} |u(t)-u(r)|_E^p&=\mathbb{E} \left|\int_0^t\int_ZS(t-s)\xi(s,z)\tilde{N}(\mathrm{d} s,\mathrm{d} z)-\int_0^r\int_ZS(r-s)\xi(s,z)\tilde{N}(\mathrm{d} s,\mathrm{d} z)\right|^p_E\\
	        &\leq 2^p\,\mathbb{E} \left|\int_r^t\int_ZS(t-s)\xi(s,z)\tilde{N}(\mathrm{d} s,\mathrm{d} z)\right|^p_E\\
	        &\hspace{2cm}+2^p\,\mathbb{E} \left|\int_0^r\int_Z\Big{(}S(t-s)-S(r-s)\Big{)}\xi(s,z)\tilde{N}(\mathrm{d} s,\mathrm{d} z)\right|^p_E
	 \end{align*}
	 \begin{align}\nonumber
	        &\leq
	        2^p\,C_p\,\mathbb{E} \int_r^t\int_Z|S(t-s)\xi(s,z)|^p_E\,\nu(\mathrm{d} z)\,\mathrm{d} s\\
\nonumber
	        &\hspace{2cm}+2^pC_p\,\mathbb{E} \int_0^r\int_Z|\big{(}S(t-s)-S(r-s)\big{)}\xi(s,z)|^p_E\,\nu(\mathrm{d} z)\,\mathrm{d} s\\
\nonumber	                    &\leq2^pC_p\,\mathbb{E} \int_0^T\int_Z1_{(r,t]}(s)|\xi(s,z)|^p_E\,\nu(\mathrm{d} z)\,\mathrm{d} s\\
\label{ineq-4.2}	        &\hspace{2cm}+2^pC_p\,\mathbb{E} \int_0^T\int_Z|1_{(0,r]}\big{(}S(t-s)-S(r-s)\big{)}\xi(s,z)|^p_E\,\nu(\mathrm{d} z)\,\mathrm{d} s.
	 \end{align}
 Let us  observe that $1_{(r,t]}(s)|\xi(s,z)|^p_E$ converges to $1_{\{t\}}(s)|\xi(s,z)|^p_E$ or $1_{\{r\}}(s) |\xi(s,z)|^p_E$ for all $(s,\omega,z)\in[0,T]\times\Omega\times Z$, as $t\searrow r$ or $r\nearrow t$.  Note also that the limit function is equal to $0$ for $\Leb \otimes \mathbb{P}\otimes \nu$ almost surely. Moreover  $1_{(r,t]}(s)|\xi(s,z)|^p_E \leq |\xi(s,z)|^p_E$ for all $(s,\omega,z)\in[0,T]\times\Omega\times Z$ and all $(r,t)\in [0,T]^2$: $r<t$.  Finally, let us note that
\[
\mathbb{E} \int_0^T\int_Z |\xi(s,z)|^p_E \,\nu(\mathrm{d} z)\,\mathrm{d} s <\infty
\]
because by assumptions, $\xi\in \mathcal{M}^p([0,T]\times Z; \hat{\mathcal{P}};E)$.\\
Hence,  by the Lebesgue DCT,
	 the first term on the right hand side of the above inequality \eqref{ineq-4.2}	 converges to $0$ as $t\searrow r$ or $r\nearrow t$. \\
For the second term, by the continuity of $C_0$-semigroup $S(t)$, $t\geq 0$, the integrand $1_{(0,r]}\big{(}S(t-s)-S(r-s)\big{)}\xi(s,z)$ converges to $0$ pointwise on $[0,T]\times\Omega\times Z$, as $t\searrow r$ or $r\nearrow t$. Moreover, by the uniform boundedness of the operators $S(t)$, $t\in [0,T]$ we infer that there exists $M>0$ such that
\begin{align*}
	     |1_{(0,r]}\big{(}S(t-s)-S(r-s)\big{)}\xi(s,\omega,z)|_E\leq M|\xi(s,\omega,z)|_E
	 \end{align*}
for all $(s,\omega,z)\in[0,T]\times\Omega\times Z$ and all $(r,t)\in [0,T]^2$ with $r<t$.

	 So,  again by invoking the  Lebesgue DCT, we deduce that the second term converges to $0$ as $t\searrow r$ or $r\nearrow t$. Therefore, we conclude the process $u$  is stochastically continuous. Hence, as the space $E$ is separable, according to Theorem $5.3$
in \cite{Wentzell}, we can find a version $\bar{u}$ of $u$ which
is separable. That is there exists a countable subset $T_0$ which is dense in $[0,T]$ such that $\bar{u}(t)$ belongs to the
set of partial limits $\{\lim_{s\in T_0,s\rightarrow t}\bar{u}(s)\}$,
for all $t\in [0,T]\backslash T_0$. Hence
\begin{align*}
    \sup_{t \geq 0}|\bar{u}(t)|=\sup_{t\in [0,T]}\lim_{s_n\to t, s_n\in T_0}|\bar{u}(s_n)|=\sup_{s_n\in T_0}|\bar{u}(s_n)|.
\end{align*}
 Since $\sup_{s_n\in T_0}|\bar{u}(s_n)|$ is $\mathcal{F}_T$-measurable, we deduce that  the function $\sup_{t \geq 0}|\bar{u}(t)|$ is also
$\mathcal{F}_T$-measurable.
\end{proof}

We now ready to embark on the proof of the main result.

\begin{proof}[Proof of Theorem \ref{theo_1}] 
Let us first notice that because the original norm $\vert \cdot\vert$ is equivalent to the norm $|\cdot|_E$, it is sufficient to prove inequality \eqref{inequality} with the latter norm, i.e.
\begin{align}\label{inequality'}
      \mathbb{E} \sup_{0 \leq s \leq t \wedge \tau }|\tilde{u}(s)|_E^{q^\prime}\leq C\ \mathbb{E} \left(\int_0^{t\wedge\tau}\int_Z|\xi(s,z)|_E^{p}\,N(\mathrm{d} s,\mathrm{d} z)\right)^{\frac{q^\prime}{p}}.
\end{align}
Secondly, let us note that it is sufficient to consider processes defined on a bounded time intervals. Hence we   fix $T>0$  and $\xi\in \mathcal{M}^p([0,T]\times Z; \hat{\mathcal{P}};E)$.
In view of Remark \ref{rem-cond 1}, let us also fix $p\in(1,2]$ and let us fix $q\in [p,\infty)$ as in Assumption \ref{assu-01}. \\
\textbf{Case I.} We first prove \eqref{inequality} for $\xi\in
\mathcal{M}^p([0,T]\times Z; \hat{\mathcal{P}};\mathcal{D}(A))$.
We have shown in Lemma \ref{lemma:strong solution} that the process
$u$ is a unique strong solution to Problem \eqref{SDE2} satisfying
\begin{align}\label{strong solution_1}
   u(t)=\int_0^tAu(s)\,\mathrm{d} s+\int_0^t\int_Z\xi(s,z)\tilde{N}(\mathrm{d} s,\mathrm{d} z),\; t\in [0,T].
\end{align}
Since the function $\phi:E\ni x\mapsto|x|^q_E$ is of $C^2$ class by assumption, one may apply the It\^{o} formula from \cite{[Hausenblas]}, see also  \cite[Theorem 3.5.3 ]{Zhu_2010_PhD thesis},   to the process $u$ given by \eqref{strong solution_1} and get, for $t \geq 0$,
\begin{align}\label{EQ-1}
   \phi(u(t))
   &=\int_0^t\phi^\prime(u(s))(Au(s))\,\mathrm{d} s+\int_0^t\int_Z\phi^\prime(u(s-))(\xi(s,z))\,\tilde{N}(\mathrm{d} s,\mathrm{d} z)\nonumber\\
  &\hspace{1cm}+\int_0^t\int_Z\Big{[}\phi(u(s-)+\xi(s,z))-\phi(u(s-))-\phi^\prime(u(s-))(\xi(s,z))\Big{]}\,N(\mathrm{d} s,\mathrm{d} z)\ \ \mathbb{P}\text{-a.s..}
\end{align}
Let $\tau\geq0$ be a stopping time.
Since by Lemma \ref{Lemma 1 for theo}, $\phi^\prime(x)(Ax)\leq 0$, for all $x\in D(A)$, we infer
that for $t \geq 0$,
\begin{align}\label{eq_5}
         \phi(u(t\wedge\tau))
   &\leq      \int_0^t\int_Z1_{(0,\tau]}(s)\phi^\prime(u(s-))(\xi(s,z))\,\tilde{N}(\mathrm{d} s,\mathrm{d} z)\nonumber\\
  &\hspace{1cm}+\int_0^{t\wedge\tau}\int_Z\Big{[}\phi(u(s-)+\xi(s,z))-\phi(u(s-))-\phi^\prime(u(s-))(\xi(s,z))\Big{]}\,N(\mathrm{d} s,\mathrm{d} z)\\
&=:I_1(t)+I_2(t)\ \ \mathbb{P}\text{-a.s..}\nonumber
\end{align}
Note that $I_1(t)$ is an $\mathbb{R}$-valued local martingale. Applying the real-valued version of Burkholder-Davis-Gundy inequality, see \cite[Proposition 15.7]{[Kallenberg]},   to the process $I_1$ we
deduce for some constant $C$ that
\begin{align}\label{eq-I_1}
    \mathbb{E} \sup_{0\leq t\leq
    T}|I_1(t)|_E&\leq C\,\mathbb{E} \left(\int_0^T\int_Z1_{(0,\tau]}(s)|\phi^\prime(u(s-))(\xi(s,z))|_E^2\,N(\mathrm{d} s,\mathrm{d} z)\right)^{\frac{1}{2}}\nonumber\\
&=C\,\mathbb{E} \Big(\sum_{t\leq T}1_{(0,\tau]}(s)|\phi^\prime(u(s-))(\xi(s,\pi(s)))|^2_E\Big)^{\frac{1}{2}} \nonumber\\
&\leq C\,\mathbb{E} \Big(\sum_{t\leq T}1_{(0,\tau]}(s)|\phi^\prime(u(s-))(\xi(s,\pi(s)))|_E^p\Big)^{\frac{1}{p}}\\
&=C\,\mathbb{E} \left(\int_0^T\int_Z1_{(0,\tau]}(s)|\phi^\prime(u(s-))(\xi(s,z))|_E^p\,N(\mathrm{d} s,\mathrm{d} z)\right)^{\frac{1}{p}} \nonumber\\
    &\leq k_1C\,\mathbb{E} \sup_{0\leq t\leq
    T\wedge\tau}|u(t)|_E^{q-1}\left(\int_0^{T\wedge\tau}\int_Z|\xi(s,z)|^p_E\,N(\mathrm{d} s,\mathrm{d} z)\right)^{\frac{1}{p}}.\nonumber
\end{align}
To estimate the integral $I_2(t)$, we observe first that
for every $t \geq 0$,
\begin{align*}
   |I_2(t)|_E&\leq\int_0^{t\wedge\tau}\int_Z\Big{|}\phi(u(s-)+\xi(s,z))-\phi(u(s-))-\phi^\prime(u(s-))(\xi(s,z))\Big{|}_E\,N(\mathrm{d} s,\mathrm{d} z)\\
       &=\sum_{s\in(0,t\wedge\tau]\cap\mathcal{D}(\pi)}\Big{|}\phi(u(s-)+\xi(s,\pi(s)))-\phi(u(s-))-\phi^\prime(u(s-))(\xi(s,\pi(s)))\Big{|}_E,\ \ \mathbb{P}\text{-a.s..}
\end{align*}
 Since assumptions the function $\phi$ is of $C^2$ class,
 by  the mean value Theorem, see \cite{[Cartan]},
  to the function $\phi$, for each $s\in [0,t\wedge\tau]$ we have
\begin{align*}
    \Big{|}\phi(u(s-)+\xi(s,\pi(s)))-\phi(u(s-))\Big{|}_E&\leq  \int_0^!\, |\xi(s,\pi(s))|_E \Big{\vert}\phi^\prime(u(s-)+\theta\xi(s,\pi(s)))\Big{\vert}_{\mathcal{L}(E)} \, d\theta.
\end{align*}
Since  $|x+\theta y|_E\leq \max\{|x|_E,|x+y|_E\}$ for all $x,y\in
E$ and $\theta \in [0,1]$, and, by Assumption \ref{assu-01},  $|\phi^\prime(x)|\leq k_1|x|_E^{q-1}$, $x\in E$, we infer that
\begin{align*}
   \Big{|}\phi^\prime(u(s-)+\theta\xi(s,\pi(s)))\Big{|}_{\mathcal{L}(E)}&\leq k_1\big{|}u(s-)+\theta\xi(s,\pi(s))\big{|}_E^{q-1}\\
   &\leq k_1\max\big{\{}|u(s-)|_E^{q-1},\big{|}u(s-)+\xi(s,\pi(s))\big{|}_E^{q-1}\big{\}}.
\end{align*}
Observe that for all $0\leq s\leq t\wedge\tau$,
\begin{align*}
|u(s-)|^{q-1}_E\leq
\sup_{0\leq r\leq t\wedge\tau}|u(r-)|^{q-1}_E\leq\sup_{0\leq t\leq
T\wedge\tau}|u(t)|^{q-1}_E.
\end{align*}
Moreover, since $u(s-)+\xi(s,\pi(s))=u(s)$, for $s\in(0,t\wedge\tau]\cap\mathcal{D}(\pi)$,  we infer that
\begin{align*}
|u(s-)+\xi(s,\pi(s))|^{q-1}_E\leq\sup_{0\leq r\leq
t\wedge\tau}|u(r)|^{q-1}_E\leq \sup_{0 \leq t \leq T \wedge \tau }|u(t)|^{q-1}_E.
\end{align*}
Therefore, we infer that for each $s\in [0,t\wedge\tau]$,
\begin{align*}
      \Big{|}\phi(u(s-)+\xi(s,\pi(s)))-\phi(u(s-))\Big{|}_E
      &\leq k_1|\xi(s,\pi(s))|_E\sup_{0 \leq t \leq T \wedge \tau }|u(t)|^{q-1}_E.
\end{align*}
It follows that
\begin{align*}
    \Big{|}\phi(u(s-)+\xi(s,\pi(s)))&-\phi(u(s-))-\phi^\prime(u(s-))(\xi(s,\pi(s)))\Big{|}_E\\
    &\leq\Big{|}\phi(u(s-)+\xi(s,\pi(s)))-\phi(u(s-))\Big{|}_E
    +\Big{|}\phi^\prime(u(s-))(\xi(s,\pi(s)))\Big{|}_E\\
    &\leq 2k_1|\xi(s,\pi(s))|_E\sup_{0 \leq t \leq T \wedge \tau }|u(t)|^{q-1}_E.
\end{align*}
On the other hand, we can also find some $0<\delta<1$ such that
\begin{eqnarray*}
  &&\hspace{-1truecm}\lefteqn{   \Big{|}\phi(u(s-)+\xi(s,\pi(s)))-\phi(u(s-))-\phi^\prime(u(s-))(\xi(s,\pi(s)))\Big{|}_E}
  \\
     &\leq &\frac{1}{2}|\xi(s,\pi(s))|^2_E|\phi^{\prime\prime}(u(s-)+\mathrm{d}elta\xi(s,\pi(s)))|
     \leq\frac{k_2}{2}|\xi(s,\pi(s))|^2_E|u(s-)+\mathrm{d}elta\xi(s,\pi(s))|_E^{q-2}.
\end{eqnarray*}
Hence, a similar argument as above, we infer that for $s\in[0,t\wedge\tau]$
\begin{align*}
\Big{|}\phi(u(s-)+\xi(s,\pi(s)))-\phi(u(s-))-\phi^\prime(u(s-))(\xi(s,\pi(s)))\Big{|}_E
\leq\frac{k_2}{2}|\xi(s,\pi(s))|^2_E\sup_{0\leq t\leq
T\wedge\tau}|u(t)|^{q-2}_E.
\end{align*}
Thus, with  $K=(2k_1)^{2-p}\big{(}\frac{k_2}{2}\big)^{p-1}$, we have
\begin{eqnarray*}
  &&\hspace{-1truecm}
  \lefteqn{
      \Big{|}\phi(u(s-)+\xi(s,\pi(s)))-\phi(u(s-))-\phi^\prime(u(s-))(\xi(s,\pi(s)))\Big{|}_E}\\
      &=&\Big{|}\phi(u(s-)+\xi(s,\pi(s)))-\phi(u(s-))-\phi^\prime(u(s-))(\xi(s,\pi(s)))\Big{|}_E^{(2-p)+(p-1)}\\
      &\leq& \left(2k_1|\xi(s,\pi(s))|_E\sup_{0\leq t\leq
      T\wedge\tau}|u(t)|^{q-1}_E\right)^{2-p}\left(\frac{k_2}{2}|\xi(s,\pi(s))|^2_E\sup_{0\leq t\leq
T\wedge\tau}|u(t)|^{q-2}_E\right)^{p-1}\\
&\leq& K|\xi(s,\pi(s))|^p_E\sup_{0 \leq t \leq T \wedge \tau }|u(t)|_E^{q-p}.
\end{eqnarray*}
 Hence, by Proposition \ref{prop-sec-2}, we get
\begin{align*}
    &\sum_{s\in(0,t\wedge\tau]\cap\mathcal{D}(\pi)}\Big{|}\phi(u(s-)+\xi(s,\pi(s)))-\phi(u(s-))-\phi^\prime(u(s-))(\xi(s,\pi(s)))\Big{|}_E\\
&\leq K \sup_{0\leq t\leq
T\wedge\tau}|u(t)|^{q-p}_E\sum_{s\in(0,t\wedge\tau]\cap\mathcal{D}(\pi)}|\xi(s,\pi(s))|^p_E
\\
    &=K\sup_{0 \leq t \leq T \wedge \tau }|u(t)|^{q-p}_E\int_0^{T\wedge\tau}\int_Z|\xi(r,z)|^p_E\
    N(\mathrm{d} r,\mathrm{d} z).
\end{align*}
Therefore, we infer that  there exists a  constant $K$  depending only  on $k_1$, $k_2$, $p$ and $q$ such that
\begin{align*}
  \mathbb{E} \sup_{t \geq 0}|I_2(t)|_E &\leq \int_0^{T\wedge\tau}\int_Z\Big{|}\phi(u(r-)+\xi(r,z))-\phi(u(r-))-\phi^\prime(u(r-))(\xi(r,z))\Big{|}_EN(\mathrm{d} r,\mathrm{d} z)\\
    &\leq K\sup_{0 \leq t \leq T \wedge \tau }|u(t)|^{q-p}_E\int_0^{T\wedge\tau}\int_Z|\xi(s,z)|^p_E\
N(\mathrm{d} s,\mathrm{d} z).
\end{align*}
Next, by  applying H\"{o}lder's and Young's inequalities to the process $I_1$, from inequality \eqref{eq-I_1} we infer that
yields
\begin{eqnarray*}
    \mathbb{E} \sup_{t \geq 0}|I_1(t)|_E&\leq& k_1C\,\left(\mathbb{E} \Big{[}\sup_{0 \leq t \leq T \wedge \tau }|u(t)|^{q-1}_E\Big{]}^{\frac{q}{q-1}}\right)^{\frac{q-1}{q}}
    \left(\mathbb{E} \left(\int_0^{T\wedge\tau}\int_Z|\xi(s,z)|_E^p\;N(\mathrm{d} s,\mathrm{d} z)\right)^{\frac{q}{p}}\right)^{\frac{1}{q}}\\
   &\leq& k_1C\,\left(\mathbb{E} \sup_{0 \leq t \leq T \wedge \tau }|u(t)|^{q}_E\right)^{\frac{q-1}{q}}
    \left(\mathbb{E} \left(\int_0^{T\wedge\tau}\int_Z|\xi(s,z)|_E^p\;N(\mathrm{d} s,\mathrm{d} z)\right)^{\frac{q}{p}}\right)^{\frac{1}{q}}\\
    &=&k_1C\,\left(\mathbb{E} \sup_{0 \leq t \leq T \wedge \tau }|u(t)|^{q}_E\ \varepsilon\right)^{\frac{q-1}{q}}
    \left(\mathbb{E} \left(\int_0^{T\wedge\tau}\int_Z|\xi(s,z)|_E^p\;N(\mathrm{d} s,\mathrm{d} z)\right)^{\frac{q}{p}}\Big{(}\frac{1}{\varepsilon}
    \Big{)}^{q-1}\right)^{\frac{1}{q}}\\
    &\leq&k_1C\,\frac{q-1}{q}\ \varepsilon\ \mathbb{E} \sup_{0 \leq t \leq T \wedge \tau }|u(t)|_E^q+k_1C\frac{1}{\varepsilon^{q-1}q}
    \mathbb{E} \left(\int_0^{T\wedge\tau}\int_Z|\xi(s,z)|^{p}_E\;N(\mathrm{d} s,\mathrm{d} z)\right)^{\frac{q}{p}}.
\end{eqnarray*}
 In the same manner for the integral $I_2(t)$ we can see that
\begin{eqnarray*}
  \mathbb{E} \sup_{t \geq 0} |I_2(t)|_E
   &\leq&K\mathbb{E} \sup_{0 \leq t \leq T \wedge \tau }|u(t)|^{q-p}_E\int_0^{T\wedge\tau}\int_Z|\xi(s,z)|^p_E\;N(\mathrm{d} s,\mathrm{d} z)\\
   &\leq&K\left(\mathbb{E} \sup_{0 \leq t \leq T \wedge \tau }|u(t)|_E^q\right)^{\frac{q-p}{q}}   \left(\mathbb{E} \left(\int_0^{T\wedge\tau}\int_Z|\xi(s,z)|^{p}_E\;N(\mathrm{d} s,\mathrm{d} z)\right)^{\frac{q}{p}}\right)^{\frac{p}{q}}\\
  &\leq&K\frac{q-p}{q}\varepsilon\,\mathbb{E} \sup_{0 \leq t \leq T \wedge \tau }|u(t)|^q_E+K\frac{p}{q}\frac{1}{\varepsilon^{\frac{q-p}{q}}}\,\mathbb{E} \left(\int_0^{T\wedge\tau}\int_Z|\xi(s,z)|^p_E\;N(\mathrm{d} s,\mathrm{d} z)\right)^{q}
   \end{eqnarray*}
Thus it  follows that
\begin{eqnarray*}
    \mathbb{E} \sup_{0 \leq t \leq T \wedge \tau }|u(t)|^q_E&\leq
&\left(k_1C\frac{q-1}{q}+K \frac{q-p}{q}\right)\varepsilon\,\mathbb{E} \sup_{0 \leq t \leq T \wedge \tau }|u(t)|_E^q\\
    &&\hspace{2cm}+\Big{(}k_1C\frac{1}{\varepsilon^{q-1}q}+K\frac{p}{q}\frac{1}{\varepsilon^{\frac{q-p}{q}}}\Big{)}\,
    \mathbb{E} \Big(\int_0^{T\wedge\tau}\int_Z|\xi(s,z)|^{p}_E\;N(\mathrm{d} s,\mathrm{d} z)\Big)^{\frac{q}{p}}.
\end{eqnarray*}
Now we can choose a suitable positive number $\varepsilon$ such that
\begin{align*}
    \left(k_1C\frac{q-1}{q}+K \frac{q-p}{q}\right)\varepsilon=\frac{1}{2}.
\end{align*}
Consequently, there exists $C$ which is independent of $A$ such that
\begin{align}\label{inequality_4}
       \mathbb{E} \sup_{0 \leq t \leq T \wedge \tau }|u(s)|^q_E\leq
C\mathbb{E} \left(\int_0^{T\wedge\tau}\int_Z|\xi(s,z)|^{p}_E\;N(\mathrm{d} s,\mathrm{d} z)\right)^{\frac{q}{p}}.
\end{align}

\textbf{Case II.} Now suppose $\xi\in
\mathcal{M}^p([0,T]\times Z; \hat{\mathcal{P}};E) $. Set
$R(n,A)=(n I-A)^{-1}$, $n\in\mathbb{N}$. Then we put
$\xi^n(t,\omega,z)=nR(n,A)\xi(t,\omega,z)$ on $[0,T]\times\Omega\times Z$. Since
$A$ is the infinitesimal generator of the $C_0$-semigroup $S(t)$,
$t\geq0$ of contractions, by the Hille-Yosida Theorem,
$\|R(n,A)\|\leq\frac{1}{n}$ and $\xi^n(t,\omega,z)\in\mathcal{D}(A)$,
for every $(t,\omega,z)\in[0,T]\times\Omega\times Z$. Moreover,
$\xi^n(t,\omega,z)\rightarrow\xi(t,\omega,z)$ pointwise on
$[0,T]\times\Omega\times Z$. Also, we observe that
$|\xi^n-\xi|=|nR(n,A)\xi-\xi|\leq 2|\xi|$. Therefore, by
applying the Lebesgue DCT, it follows  that $\mathbb{P}$-a.s.
\begin{align*}
   \mathbb{E} \int_0^T\int_Z1_{(0,\tau]}(s)\, |\xi^n(s,z)-\xi(s,z)|_E^p\;\nu(\mathrm{d} z)\mathrm{d} s \;\to \;0 \mbox{ as } n\rightarrow\infty.
\end{align*} Since
the Poisson random measure $N$ is a $\mathbb{P}$-a.s.  positive and
\begin{align*}
   \mathbb{E} \int_0^T\int_Z1_{(0,\tau]}(s)\,|\xi^n(s,z)-\xi(s,z)|^p_E\;N(\mathrm{d} s,\mathrm{d} z)=\mathbb{E} \int_0^T\int_Z1_{(0,\tau]}|\xi^n(s,z)-\xi(s,z)|^p_E\;\nu(\mathrm{d} z)\mathrm{d} s,
\end{align*}
we see that $\mathbb{P}\text{-a.s.}$
\begin{align*}
      \int_0^T\int_Z1_{(0,\tau]}(s)\,|\xi^n(s,z)-\xi(s,z)|^p\;N(\mathrm{d} s,\mathrm{d} z)\rightarrow0,\ \
      \text{as }n\rightarrow\infty.
\end{align*}
Let us fix $n\in\mathbb{N}$.  Clearly, $\xi^n\in\mathcal{M}^p([0,T]\times
Z;\mathcal{D}(A))$ and so  we may define a process $u^n$ by
\begin{align*}
    u^n(t)=\int_0^tS(t-s)\xi^n(s,z)\,\tilde{N}(\mathrm{d} s,\mathrm{d} z),\; t\in [0,T].
\end{align*}
 Since by Lemma \ref{lemma:strong solution}, the process $u^n$ is a  strong solution  of equation \eqref{SDE} with the process $\xi$ replaced by $\xi^n$,  we infer that it  is an $E$-valued
c\`{a}dl\`{a}g. Hence by inequality \eqref{inequality_4},  for every  stopping time $\tau\geq0$ the following inequality holds
\begin{align*}
           \mathbb{E} \sup_{0 \leq t \leq T \wedge \tau }|u^n(t)|^q\leq C\,\mathbb{E} \left(\int_0^{T\wedge\tau}\int_Z|\xi^n(s,z)|^p_E\,N(\mathrm{d} s,\mathrm{d} z)\right)^{\frac{q}{p}}.
\end{align*}
On the other hand, since by inequality \eqref{Ito-isometry}, we have
\begin{align}\label{eq-a1}
    \mathbb{E} |u^n(t)-u(t)|_E^p
   &=\mathbb{E} \left|\int_0^t\int_Z\Big{(}S(t-s)\xi^n(s,z)-S(t-s)\xi(s,z)\Big{)}\,\tilde{N}(\mathrm{d} s,\mathrm{d} z)\right|_E^p\nonumber\\
   &\leq C_p\,\mathbb{E} \int_0^T\int_Z|\xi^n(s,z)-\xi(s,z)|^p_E\,\nu(\mathrm{d} z)\,\mathrm{d} s,\; t\in [0,T],
\end{align}
we deduce that for
 $t \in [0,T]$,  $u^n(t)$ converges to $u(t)$ in $L^p(\Omega)$. Moreover, since according to \eqref{inequality_4},
\begin{align*}
    \mathbb{E} \sup_{t \geq 0}|u^n(t)-u^m(t)|^q_E\leq C\mathbb{E} \left(\int_0^T\int_Z|\xi^n(s,z)-\xi^m(s,z)|^{p}_E\,N(\mathrm{d} s,\mathrm{d} z)\right)^{\frac{q}{p}}
\end{align*}
 and $\xi^n(t,\omega,z)\rightarrow\xi(t,\omega,z)$  on $[0,T]\times\Omega\times Z$, we infer that the right hand-side of last
inequality converges to $0$ as $n,m\rightarrow\infty$. Hence,
it is possible to choose a sequence $\{n_k\}_{k=1}^{\infty}$ of natural numbers  such that
\begin{align*}
    \mathbb{E} \sup_{t \geq 0}|u^{n_{k+1}}(t)-u^{n_k}(t)|_E^q<\frac{1}{k^{2q+2}}.
\end{align*}
Hence, on the basis of Chebyshev inequality, we obtain
\begin{align*}
       \mathbb{P}\left\{\sup_{t \geq 0}|u^{n_{k+1}}(t)-u^{n_k}(t)|_E>\frac{1}{k^2}\right\}\leq k^{2q}\mathbb{E} \sup_{t \geq 0}|u^{n_{k+1}}(t)-u^{n_k}(t)|_E^q<\frac{1}{k^2}.
\end{align*}
Thus the series $\sum_{k=1}^{\infty}\mathbb{P}\left\{\sup_{0\leq
t\leq T}|u^{n_{k+1}}(t)-u^{n_k}(t)|_E>\frac{1}{k}\right\}$ is convergent. It follows from the Borel-Cantelli Lemma that with
probability $1$ there exists an integer $k_0$ such that
\begin{align*}
       \sup_{t \geq 0}|u^{n_{k+1}}(t)-u^{n_k}(t)|_E\leq \frac{1}{k^2}\text{, for all }k\geq k_0.
\end{align*}
Consequently, the series of c\`{a}dl\`{a}g processes
\begin{equation*}
\sum_{k=1}^{\infty}[u^{n_{k+1}}(t)-u^{n_k}(t)],\;\in [0,T]
\end{equation*}
 converges, $\mathbb{P}$-a.s.,
uniformly on $[0,T]$,  to a c\`{a}dl\`{a}g
process  which we shall denote by
$\tilde{u}=(\tilde{u}(t))_{t \geq 0}$. In view of Lemma \ref{lem:separability}, it is possible to assume that the process $\tilde{u}$ is separable. Thus, the function $\sup_{t \geq 0}|\tilde{u}(t)|^q$ is also measurable. Moreover, we have
 \begin{align}\label{eq-a2}
     \mathbb{E} \sup_{t \geq 0}|u^{n_k}(t)-\tilde{u}(t)|_E^q\rightarrow 0,\ \ \ \text{as }n_k\rightarrow\infty.
 \end{align}
 Therefore, by the Minkowski Inequality and inequality \eqref{inequality_4}, we have
 \begin{align*}
      \Big[\mathbb{E} \sup_{0\leq s\leq T\wedge\tau}|\tilde{u}(t)|_E^q\Big]^{\frac{1}{q}}&\leq \Big[\mathbb{E} \sup_{0 \leq t \leq T \wedge \tau }|\tilde{u}(t)-u^{n_k}(t)|_E^q\Big]^{\frac{1}{q}}+\Big[\mathbb{E} \sup_{0 \leq t \leq T \wedge \tau }|u^{n_k}(t)|_E^q\Big]^{\frac{1}{q}}\\
      &\leq \Big[ \mathbb{E} \sup_{0 \leq t \leq T \wedge \tau }|\tilde{u}(t)-u^{n_k}(t)|_E^q\Big]^{\frac{1}{q}}+\Big[C\,\mathbb{E} \Big(\int_0^{T\wedge\tau}\int_Z|\xi^{n_k}(s,z)|^{p}_E\,N(\mathrm{d} s,\mathrm{d} z)\Big)^{\frac{q}{p}}\Big]^{\frac{1}{q}}.
 \end{align*}
 Note that the constant $C$ on the right hand side of the above inequality does not depend on operator $A$. So the constant $C$ remains the same
  for every $n$. It follows by letting $n_k\rightarrow\infty$ in above inequality that
  \begin{align*}
          \mathbb{E} \sup_{0 \leq t \leq T \wedge \tau }|\tilde{u}(t)|^q\leq C\, \mathbb{E} \Big(\int_0^{T\wedge\tau}\int_Z|\xi(s,z)|^{p}_E\,N(\mathrm{d} s,\mathrm{d} z)\Big)^{\frac{q}{p}}.
  \end{align*}
  Also, by Minkowski inequality we have for every $t \geq 0$,
\begin{align*}
   \left(\mathbb{E} |\tilde{u}(t)-u(t)|^p_E\right)^{\frac{1}{p}}&\leq \left(\mathbb{E} |\tilde{u}(t)-u^{n_k}(t)|^p_E\right)^{\frac{1}{p}}+\left(\mathbb{E} |u(t)-u^{n_k}(t)|^p_E\right)^{\frac{1}{p}}\\
&\leq \left(\mathbb{E} |\tilde{u}(t)-u^{n_k}(t)|^q_E\right)^{\frac{1}{q}}+\left(\mathbb{E} |u(t)-u^{n_k}(t)|^p_E\right)^{\frac{1}{p}}\\
&\leq \left(\mathbb{E} \sup_{0\leq t\leq
T}|\tilde{u}(t)-u^{n_k}(t)|^q_E\right)^{\frac{1}{q}}+\left(\mathbb{E} |u(t)-u^{n_k}(t)|^p_E\right)^{\frac{1}{p}}.
\end{align*}
Letting $n\rightarrow\infty$, it follows from \eqref{eq-a1} and \eqref{eq-a2} that $u(t)=\tilde{u}(t)$ in
$L^p(\Omega)$ for any $t \geq 0$. This shows the inequailty
\eqref{inequality} for $q^\prime=q$. The case $q^\prime>q$ follows from the fact
that if Banach space $E$ satisfies Assumption
\ref{assu-01} for some $q$, then Condition $1$ is also satisfied with
$q^\prime>q$.

\end{proof}

The following result could be derived immediately from the proof of
the above result.

\begin{coro}Let $E$ be a Banach space satisfying Assumption \ref{assu-01}. Then the stochastic convolution
process $u$ has a c\`{a}dl\`{a}g modification.
\end{coro}


\begin{coro}\label{coro_theo_2}
\label{highq} Let $E$ be a Banach space
satisfying Assumption \ref{assu-01}. Let $\sqrt{2}\leq p\leq 2$. Then for any $n\in \mathbb{N}$ with $p^n\geq q$,
there exists a constant $C=C(E,n)$ such that for every
$\xi\in \bigcap_{k=1}^n\mathcal{M}^{p^k}_{\loc}([0,\infty)\times Z; \hat{\mathcal{P}};E)
 $ and for
every stopping time $\tau>0$ and $t\geq 0$,
\begin{align}\label{inequality_13}
      \mathbb{E} \sup_{0\leq s\leq t \wedge \tau }|\tilde{u}(s)|_E^{p^n}\leq C\
     \sum_{k=1}^n{\mathbb{E} }\left(\int_0^{t \wedge \tau}\int_Z |\xi(s,z)|_E^{p^k}\,\nu (\mathrm{d} z) \mathrm{d} s\right)^{{p^{n-k}}},
\end{align}
where $\tilde{u}$ is a separable and c\`{a}dl\`{a}g modification of $u$ as
before.
\end{coro}
The proof of Corollary \ref{coro_theo_2}  is similar to the proof
Lemma 5.2 in Bass and Cranston \cite{[Bass+Cran_1986]} or of Lemma
4.1 in Protter and Talay \cite{[Talay_1997]}. The two essential ingredients
of that proof are formulated below. The latter one is
about integration  of real valued processes.

\begin{lemma}\label{claim-2} Let $E$ be a martingale type $p$ Banach space, $1<p\leq 2$, satisfying Assumption \ref{assu-01}. Let $\tau>0$ be a  stopping time.
For any $q^\prime \geq q$, there exists a constant $C$ such that,
for
all $\xi\in \mathcal{M}^p_{\loc}([0,\infty)\times Z; \hat{\mathcal{P}};E) $, we have

\begin{equation}\label{main-ineq}
\mathbb{E} \sup_{0\leq s\leq t\wedge\tau}\left| \int_0^s\int_Z
\xi(r,z)\,\tilde N(\mathrm{d} r,\mathrm{d} z)\right|_{E}^{q^\prime}\le C\, \mathbb{E} \left( \int_0^{t\wedge\tau}\int_Z
|\xi(s,z)|_E ^p\,
             N(\mathrm{d} s,\mathrm{d} z)\right)^\frac {q^\prime}{p}, \; t \geq 0.
\end{equation}
\end{lemma}

\begin{proof}[Proof of Lemma \ref{claim-2}] This result is a special case of Theorem \ref{theo_1}  with $S(t)=I$, $t \geq 0$.
\end{proof}
\begin{lemma}\label{claim-1}
Let $\sqrt{2}\leq p\leq 2$. For any $n\in\mathbb{N}$ there exists  a constant  $D_n>0$ such that for any
process $$f\in\bigcap_{k=1}^n \mathcal{M}^{p^k}_{\loc}([0,\infty)\times Z; \hat{\mathcal{P}};\mathbb{R}), $$
and all $t \geq 0$ and stopping times $\tau>0$,  the following inequality holds
\begin{equation}\label{inequality_14}
    \mathbb{E} \sup_{0\leq s\leq t\wedge\tau}\left|\int_0^{s}\int_Z f(r,z)\,\tilde N (\mathrm{d} r,\mathrm{d} z)\right|^{{p^{n}}}
    \leq D_n
     \sum_{k=1}^n \mathbb{E} \left(  \int_0^{t\wedge\tau}\int_Z|f(s,z)|^{p^k}\,\nu(\mathrm{d} z)  \mathrm{d} s\right)^{p^{n-k}}.
     \end{equation}
\end{lemma}

\begin{proof}[Proof of Lemma \ref{claim-1}] We shall show this Lemma by induction. The case $n=1$ follows from \cite{[Brz+Haus_2009]}. Now we assume that the assertion in the Claim is true for $n-1$, where $n\in\mathbb{N}$ and $n\geq 2$. We will show that it is still true for $n$.
Since by assumption $f\in\mathcal{M}^p([0,T]\times
Z;
\hat{\mathcal{P}};\mathbb{R})$,
so both integrals $\int_0^t\int_Z |f(s,z)|^pN(ds,dz)$ and
$\int_0^t\int_Z|f(s,z)|^p\nu(dz)ds$ are well defined as
Lebesgue-Stieltjes integrals. Moreover, we have, for every $t\geq 0$ and stopping time $\tau>0$,
 \begin{align}\label{eq-12}
 \int_0^{t\wedge\tau}\int_Z |f(s,z)|^p\,\tilde{N}(\mathrm{d} s,\mathrm{d} z)=\int_0^{t\wedge\tau}\int_Z |f(s,z)|^p\,N(\mathrm{d} s,\mathrm{d} z)-\int_0^{t\wedge\tau}\int_Z|f(s,z)|^p\,\nu(\mathrm{d} z)\mathrm{d} s\ \ \mathbb{P}\text{-a.s..}
 \end{align}
Since for $n\geq 2$, $2\leq p^n\leq 2^n$, so the function $|\cdot|_{\mathbb{R}}^{p^n}$ satisfies Assumption \ref{assu-01} with $q=p^n$.
Hence by using first the inequality \eqref{main-ineq} and then inequality \eqref{eq-12},  we infer \begin{eqnarray}\label{hmus} \nonumber
\lefteqn{  \mathbb{E} \sup_{0\leq s\leq {t\wedge\tau}}\left|\int_0^{s}\int_Z f(r,z)\tilde N
(\mathrm{d} r,\mathrm{d} z)\right|^{{p^{n}}}
    \leq C\,\mathbb{E} \left|\int_0^{t\wedge\tau}\int_Z |f(s,z)|^{p}\, N (\mathrm{d} s,\mathrm{d} z)\right|^{{p^{n-1}}}
} &&
\\\nonumber
&\leq & 2^{p^{n-1}}C \,\Big\{
 \mathbb{E}  \left(
       \int_{0 } ^{t\wedge\tau}\int_Z |f (s,z)|^p  \;\tilde N  (\mathrm{d} s,\mathrm{d} z)  \right)^{p ^{n-1}}
 + \mathbb{E}  \left( \int_{0 } ^{t\wedge\tau}\int_Z  |f (s,z)|^p  \;\nu (\mathrm{d} z) \, \mathrm{d} s\right)^{p ^{n-1}}\Big\}
\end{eqnarray}

Next, by applying the induction assumption to the real valued process
$$|f|^p\in  \bigcap_{k=1}^{n-1}\mathcal{M}^{p^k}_{\loc}([0,\infty)\times Z; \hat{\mathcal{P}};\mathbb{R})
,$$
we get
\begin{eqnarray}\label{hmus2}\nonumber
&&\hspace{-1truecm}\lefteqn{ \mathbb{E} \sup_{0\leq s\leq {t\wedge\tau}}\left|\int_0^{s}\int_Z f(r,z)\tilde N
(\mathrm{d} r,\mathrm{d} z)\right|^{{p^{n}}}}
\\
\nonumber
&\leq & 2^{p^{n-1}}C \, \Big( D_{n-1}\sum_{i=1}^{n-1} \mathbb{E}  \Big(
       \int_{0 } ^{t\wedge\tau}\int_Z |f (s,z)|^{p^{i+1}}  \; \nu (\mathrm{d} z) \, \mathrm{d} s  \Big)^{p ^{n-1-i}}
\\
\nonumber & +& \mathbb{E}  \Big( \int_{0 } ^{t\wedge\tau}\int_Z  |f (s,z)|^p  \;\nu (\mathrm{d} z) \, \mathrm{d} s\Big)^{p ^{n-1}}
\Big)
\leq D_n\sum_{k=1}^{n}\mathbb{E}   \Big(
       \int_{0 } ^{t\wedge\tau}\int_Z |f (s,z)|^{p^{k}}  \; \nu (\mathrm{d} z) \, \mathrm{d} s \Big)^{p^{n-k}} .
\end{eqnarray} This proves the validity of the assertion in the Lemma  for
$n$ which completes the whole proof.
\end{proof}

\begin{proof}[Proof of Corollary \ref{coro_theo_2}] Let us take $n\in\mathbb{N}$. By applying first Theorem \ref{theo_1} and next the equality \eqref{eq-12}
when $\xi\in\mathcal{M}^p([0,T]\times Z; \hat{\mathcal{P}};E)$, we deduce
that for all $t\in [0,T]$, \begin{eqnarray} \nonumber \mathbb{E}  \sup_{0\le s\le t} |
\tilde u(s) | _E^{p^{n}} &\leq& C\: \mathbb{E}  \left(  \int_{0} ^t \int_Z\left|
\xi(s,z)\right|_E ^p  N(\mathrm{d} s,\mathrm{d} z )\right) ^ {p^{n-1}}
\\
&\leq& 2^{p^{n-1}}C  \,
 \mathbb{E}  \left(
  \int_{0 } ^t\int_Z |\xi (s,z)|_E^p  \;\tilde N  (\mathrm{d} s,\mathrm{d} z )  \right)^{p ^{n-1}}\nonumber\\
 && +  2^{p^{n-1}}C \  \mathbb{E} \left( \int_{0 } ^t\int_Z  |\xi (s,z)|_E^p  \;\nu (\mathrm{d} z) \, \mathrm{d} s\right)^{p ^{n-1}}\nonumber\\
 &\leq&2^{p^{n-1}}\ C  \,
       D_{n-1}\sum_{k=1}^{n-1}\mathbb{E}    \left(\int_{0 } ^t\int_Z |\xi (s,z)|_E^{p^{k+1}}  \nu(\mathrm{d} z)\,\mathrm{d} s\right)^{p ^{n-1-k}}\nonumber\\
&& + 2^{p^{n-1}}\ C \ \mathbb{E}   \  \left( \int_{0 } ^t\int_Z  |\xi (s,z)|_E^p  \;\nu (\mathrm{d} z) \, \mathrm{d} s\right)^{p ^{n-1}}\nonumber\\
&\leq&C(n)\sum_{k=1}^n\mathbb{E} \left(\int_{0 } ^t\int_Z |\xi (s,z)|_E^{p^{k}}
\nu(\mathrm{d} z)\,\mathrm{d} s\right)^{p ^{n-k}},\nonumber \end{eqnarray} where we used in the third
inequality Lemma \ref{claim-2} with $f$ replaced by real-valued
process $\xi$ such that
$$|\xi|^p_E\in
\bigcap_{k=1}^{n-1}\mathcal{M}^{p^k}_{\loc}([0,\infty)\times Z; \hat{\mathcal{P}};\mathbb{R})
.
$$
This completes the proof of Corollary \ref{coro_theo_2}.
\end{proof}

\section{Extension to progressively measurable integrands}\label{sec-pm}

Corollary \ref{coro_theo_2} can be generalized to integrands which are progressively measurable processes.  Let us recall that
a process $\xi:[0,T]\times \Omega\times Z\to E$ is $\mathbb{F}\otimes \mathcal{Z}$-progressively measurable, if $\xi$ is
$\mathcal{BF}\otimes \mathcal{Z}/\mathcal{B}(E)$-measurable, where, see \cite[section 6.5]{Wentzell}, $\mathcal{BF}$ is the $\sigma$-field consisting of all sets $A\subset  [0,T]\times \Omega$ such that for every $t\in [0,T]$, the set $A\cap \big( [0,t]\times \Omega)$ belongs to the sigma field $\mathcal{B}_{[0,t]}\otimes \mathcal{F}_t$. Note that $\mathcal{BF}\otimes \mathcal{Z}$ is the $\sigma$-field generated by a family  of all sets $A\subset  [0,T]\times \Omega \times Z$ such that for every $t\in [0,T]$, the set $A\cap \big( [0,t]\times \Omega \times Z)$ belongs to the sigma field $\mathcal{B}_{[0,t]}\otimes \mathcal{F}_t \otimes \mathcal{Z}$.\\
For $p\in [1,\infty)$, the set  of all  of $p$--integrable $\mathcal{BF}\otimes \mathcal{Z}$-progressively processes $\xi:[0,T]\times \Omega\times Z\to E$ will be denoted by \[\mathcal{M}^{p}([0,T]\times Z; \mathcal{BF}\otimes \mathcal{Z};E)\] and the Banach space of
 all equivalence classes of   $p$--integrable $\mathcal{BF}\otimes \mathcal{Z}$-progressively processes $\xi:[0,T]\times \Omega\times Z\to E$ will be denoted by \[\mathbb{M}^{p}([0,T]\times Z; \mathcal{BF}\otimes \mathcal{Z};E).\]

As noted in Remark \ref{rem-predictable versus progressively measurable}, the It\^o integral with respect to a compensated Poisson random measure of  processes from the class has been introduced in \cite{[Brz+Haus_2009]}, see also \cite[Theorem 3.2.27]{Zhu_2010_PhD thesis}.
The following follows from
\cite[Theorem 3.2.27]{Zhu_2010_PhD thesis}.

\begin{proposition}\label{prop-pm}
 If  $p\in [1,\infty)$  and a  progressively measurable    process $\xi:[0,T]\times \Omega\times Z\to E$ belongs to  $ \mathcal{M}^{p}([0,T]\times Z; \mathcal{BF}\otimes \mathcal{Z};E)$
then
there exists a sequence of c\`{a}gl\`{a}d step functions $\xi_n\in
\mathcal{M}^p_{step}([0,T]\times Z;\hat{\mathcal{P}};E)$, such that $\xi_n\to \xi$ in $\xi \in \mathcal{M}^{p}([0,T]\times Z; \mathcal{BF}\otimes \mathcal{Z};E)$, as $n\rightarrow\infty$.
\end{proposition}

\begin{coro}\label{coro_theo_progr}
\label{highq-1} Let $E$ be a Banach space
satisfying Assumption \ref{assu-01}. Let $\sqrt{2}\leq p\leq 2$ and  $n\in \mathbb{N}$ such that  $p^n\geq q$. Then for for every
$\xi\in \cap_{k=1}^n\mathcal{M}^{p^k}([0,T]\times Z; \mathcal{BF}\otimes \mathcal{Z};E)
 $ there exists a process $\tilde{u}$ which is  a separable and c\`{a}dl\`{a}g modification of the stochastic convolution process $u$ defined, as before, by \eqref{sto_convolution}. Moreover,
there exists a constant $C=C(E,n)$ independent of $\xi$, such that
for every  stopping time $\tau$ and $t\in[0,T]$,  inequality \eqref{inequality_13} holds true.
\end{coro}
\begin{proof} By Proposition \ref{prop-pm}, there exists a sequence
$\{ \xi_i:i\in\mathbb{N}\}\subset \cap_{k=1}^ n\mathcal{M}^{p^k}_{step}([0,T]\times Z;\hat{\mathcal{P}};E)$ of c\`{a}gl\`{a}d processes
 convergent  to $\xi$ in $\cap_{k=1}^n\mathcal{M}^{p^k}([0,T]\times Z; \mathcal{BF}\otimes \mathcal{Z};E)$.
By Theorem \ref{theo_1}, for every $i$, the exists a  separable c\`{a}dl\`{a}g
 modification $\tilde{u}_i$ of the process $u_i$
being the solution of the Problem
\begin{align}\label{sc-01l}
   u_i(t)=\int_0^t\int_ZS(t-s)\xi_i(s,z)\tilde{N}(\mathrm{d} s,\mathrm{d} z),\ t\in[0,T].
\end{align}
which satisfies
\begin{align}\label{inequality_13_n}
      \mathbb{E} \sup_{0\leq s\leq t \wedge \tau }|\tilde{u}_i(s)|_E^{p^n}\leq C\
     \sum_{k=1}^n{\mathbb{E} }\left(\int_0^{t \wedge \tau}\int_Z |\xi_i(s,z)|_E^{p^k}\,\nu (\mathrm{d} z) \mathrm{d} s\right)^{{p^{n-k}}},\ t\in[0,T], \quad l\in\mathbb{N}.
\end{align}
and, for all $i,j\in \mathbb{N}$,
\begin{align}\label{inequality_14_n}
      \mathbb{E} \sup_{0\leq s\leq t \wedge \tau }|\tilde{u}_i(s)-\tilde{u}_j|_E^{p^n}\leq C\
     \sum_{k=1}^n{\mathbb{E} }\left(\int_0^{t \wedge \tau}\int_Z |\xi_i(s,z)-\xi_j(s,z)|_E^{p^k}\,\nu (\mathrm{d} z) \mathrm{d} s\right)^{{p^{n-k}}},\ t\in[0,T], \quad l\in\mathbb{N}.
\end{align}
Arguing as in the proof of Theorem \ref{theo_1} we can conclude the proof.
 \end{proof}

\section{Final comments}

Inequality \eqref{SC} can also be derived by the method used by the third  named author and Seidler in  \cite{[Haus+Seidler]}, see as inequality (4)  therein. These authors used the Szek\"{o}falvi-Nagy's Theorem on unitary dilations in Hilbert spaces. However, this method  works only for  analytic  semigroups of contraction type while  the results from the current paper are valid for all  $C_0$ semigroups of contraction type. Let us now formulate the following result whose proof is a clear combination of the proofs from \cite{[Haus+Seidler]} and \cite{[Fr+Weis]}. For the explanation of the terms used we refer the reader to the latter work. Similar observation for processes driven by a Wiener process was made independently by Seidler \cite{seidler}.

\begin{theorem}\label{dil}
Let $E$ be a martingale type $p$ Banach space, where $1<p\leq 2$.
Let $-A$ be a generator of a bounded analytic semigroup in $E$ such that for some $\theta<\tfrac12\pi$, the operator
 $A$ has a bounded $H ^ \infty(S_\theta)$ calculus.
Then,
for any $0<q^\prime<\infty$, there exists a constant $C$ such that
for
all $\xi\in \mathcal{M}^p_{\loc}([0,\infty)\times Z; \hat{\mathcal{P}};E)$ and for
every  stopping time $\tau>0$, we have

\begin{eqnarray}\label{main-ineq2}
\mathbb{E} \sup_{0\leq s\leq t  \wedge \tau}\left| \int_0^s\int_Z S(s-r)
\xi(r,z)\,\tilde N(\mathrm{d} r,\mathrm{d} z)\right|_{E}^{q^\prime}\le C\, \mathbb{E} \left( \int_0^{t \wedge \tau} \int_Z
|\xi(s,z)|_E ^p\,
             N(\mathrm{d} s,\mathrm{d} z)\right)^\frac {q^\prime}{p}, \quad t \geq 0.
\end{eqnarray}
\end{theorem}

The following result could be derived immediately from the proof of
above theorem.

\begin{coro}Let  $E$ be a martingale type $p$ Banach space, where $1<p\leq 2$. Let $-A$ be a generator of a bounded analytic semigroup in $E$ such that for some $\theta<\tfrac12\pi$ the operator $A$ has a bounded $H ^ \infty(S_\theta)$ calculus.
Then,
the stochastic convolution
process $u$ defined by \eqref{SC} has c\`{a}dl\`{a}g modification.
\end{coro}

\appendix

\section{Appendix}\label{sec-appendix}

\begin{defi}
A Banach space $E$ with norm $\|\cdot\|$ is of martingale type $p$, for $p\in (1,2]$ if and only if there exists a constant $C_p(E)>0$ such that for any $E$-valued discrete martingale $\{M_k\}_{k=1}^n$ the following inequality holds
\begin{align}
      \mathbb{E} \|M_n\|^p\leq C_p(E)\sum_{k=0}^n\mathbb{E} \|M_k-M_{k-1}\|^p,
\end{align}
with $M_{-1}=0$ as usual.
\end{defi}
\begin{remark}
Any separable Hilbert space is of martingale type $2$  with
\begin{align*}
      \mathbb{E} \|M_n\|^2=\sum_{k=0}^n\mathbb{E} \|M_k-M_{k-1}\|^2.
\end{align*}
 If $E$ and $F$ are isomorphic Banach spaces, then $E$ is of martingale type $p$ if and only if $F$ is of martingale type $p$.
\end{remark}
The following definition of $2$-smooth Banach spaces in terms of asymptoticity of the modulus of smoothness of the norm can be found in \cite{Pisier} and \cite{Pisier86}.
\begin{defi} A Banach space $E$ is $p$-smooth if there exists an equivalent norm defined by the modulus of smoothness of $(E,\|\cdot\|)$
\begin{align*}
     \rho_E(t)=\sup\{\frac{1}{2}(\|x+ty\|+\|x-ty\|)-1:\|x\|=\|y\|=1\}
\end{align*}
satisfying $\rho_E(t)\leq Kt^p$ for all $t>0$ and some $K>0$.
\end{defi}

\begin{remark}A Banach space is of martingale type $p$ if and only if it is $p$-smooth, see \cite{Pisier}.  Hence all spaces $L^q(\mu)$, for $q\in[p,\infty)$ and $q>1$ with an arbitrary positive measure $\mu$ are of martingale type $p$. Note that any closed subspaces of martingale type $p$ spaces are of martingale type $p$. So the Sobolev spaces $W^{k,q}$, for $q\in[p,\infty)$ and $k>0$ are of martingale type $p$.
\end{remark}
The following Lemma can be found in \cite{[Neerven+Zhu]}.
\begin{lemma}\label{lem-p-smooth}
A Banach space $E$ is $p$-smooth, $1<p\leq 2$, if and only if the Fr\'{e}chet derivative of the norm function $x\mapsto\n{x}^p$ is globally $(p-1)$-H\"{o}lder continuous on $E$.
\end{lemma}

\begin{lemma}\label{lem-7.6} If a real separable Banach space $E$ satisfies Assumption \eqref{assu-01}, then $E$ is of martingale type $p$, for all $p\in(1,2]$.
\end{lemma}

\begin{proof}[Proof of Lemma \ref{lem-7.6}] see \cite{Pisier86} It is sufficient to consider the case $p=2$, see \cite{Pisier86}. Let $E$ be a Banach space with norm $\n{\cdot}$. We assume that $q>2$ and that the function
\[
\psi:E \ni x \mapsto \n{x}^q\in \mathbb{R}
\]
is of $C^2$-class and satisfies the standard assumptions, i.e.
\[
\n{\psi^\prime(x)} \leq C_1 \n{x}^{q-1}, \;\;
\n{\psi^{\prime\prime}(x)} \leq C_2 \n{x}^{q-2}, \, x \in E.
\]
We consider a function
\[
\phi:E \ni x \mapsto \n{x}^q\in \mathbb{R}.
\]
We claim that $\phi$ is of $C^1$ class and of $C^2$ class on $E\setminus\{0\}$, and $\phi^\prime$ is globally Lipschitz continuous on $E$.\\
To see this, observe first by chain rule that for any $x\in E\setminus\{0\}$,
\[
\phi^\prime(x)=\frac{2}{q}\big[ \psi(x)\big]^{\frac{2}{q}-1} \psi^\prime(x).
\]
Thus,
\[
\n{\phi^\prime(x)} \leq C \big(  \n{x}^{q} \big)^{\frac{2}{q}-1} \n{x}^{q-1}=C \n{x}.
\]
In particular, $\lim_{x\to 0}\n{\phi^\prime(x)} \to 0$  and thus $\phi$ is differentiable at $0$ and $d_0\phi=0$.

Applying the chain rule again, we have for $x\in E\setminus\{0\}$
\begin{equation*}
\phi^{\prime\prime}(x)=
\frac{2}{q}(\frac{2}{q}-1)\big[ \psi(x)\big]^{\frac{2}{q}-2} \psi^\prime(x)\otimes \psi^\prime(x)
+\frac{2}{q}\big[ \psi(x)\big]^{\frac{2}{q}-1} \psi^{\prime\prime}(x)
\end{equation*}
As above, using the assumptions of the derivatives of $\psi$ we infer that there exists $C>0$ such that
\[
\n{\phi^{\prime\prime}(x)} \leq C, \;\; x \in E\setminus\{0\}.
\]
For any $x,y\in E\setminus\{0\}$, by applying the mean value Theorem, see e.g. \cite{[Cartan]}, we have
\begin{align*}
     \n{\phi^{\prime}(x)-\phi^{\prime}(y)}=\n{\phi^{\prime\prime}(\theta)(x-y)}\leq C\n{x-y},
\end{align*}
where the point $\theta$ lies on the same line segment between $x$ and $y$. Hence the first derivative $\phi^{\prime}$ is globally Lipschitz continuous. By applying Lemma \ref{lem-p-smooth}, we infer that the Banach space $E$ is $2$-smooth and  hence it is of martingale type $2$.
\end{proof}

\begin{acknowledgements}
Preliminary versions of this work were presented at the First
CIRM-HCM Joint Meeting on Stochastic Analysis and SPDE's which was
held at Trento (January 2010). The research of the first named
author was partially supported by an ORS award at the University of
York. Results presented in this  article are included in  the
PhD thesis of the first named author.  This work was supported by
the FWF-Project P17273-N12. Part of the work was done at the Newton
Institute for Mathematical Sciences in Cambridge (UK), whose support is gratefully acknowledged, during the
    program "Stochastic Partial Differential Equations". The second named author wishes to thank Clare Hall (Cambridge)  for hospitality.
    The first and second named authors wish to thank University of Salzburg  for hospitality.
   Finally, the authors acknowledge that the comments and suggestions of Anna Chojnowska-Michalik made for the PhD thesis of the first named author have also influenced the final presentation of this paper.
   The authours would like to thank an anonymous referee and the Associated Editor for the useful comments which greatly enhanced the quality of the presentation.
\end{acknowledgements}


\begin{thebibliography}{ADK1}


\bibitem{[Applebaum]}D. Applebaum, \textit{L\'{e}vy processes and stochastic calculus}, Cambridge Studies in Advanced Mathematics, \textbf{93}, Cambridge University Press, Cambridge, 2004.


\bibitem{[Bass+Cran_1986]} R.~F. Bass and M.~Cranston,  \textit{The Malliavin calculus for pure jump processes and applications to
  local time}, { Ann. Probab.} \textbf{14}, no. 2, 490--532 (1986).


\bibitem{[Brz+Haus_2009]} Z. Brze\'{z}niak and E.  Hausenblas, \textit{Maximal regularity for stochastic convolutions driven by L\'evy processes},  Probab. Theory Related Fields  \textbf{145},  no. 3-4, 615--637   (2009)

\bibitem{[Brz+Liu+Zhu]} Z. Brze\'{z}niak, W. Liu and J. Zhu,\textit{Strong solutions for SPDE with locally monotone coefficients
driven by L\'{e}vy noise}, Nonlinear Analysis: Real World Applications \textbf{17}, 283-310, (2014)

\bibitem{[Brz+Pesz_2000]} Z. Brze\'{z}niak and S. Peszat, \textit{Maximal inequalities and exponential estimates for stochastic convolutions in Banach spaces},  Stochastic processes, physics and geometry: new interplays, I (Leipzig, 1999),  55--64, CMS Conf. Proc., \textbf{28}, Amer. Math. Soc., Providence, RI, (2000)
\bibitem{[Brz+Pesz_2001]}
Z. Brze\'zniak, S.  Peszat,
\newblock \textit{Stochastic two dimensional Euler equations},
\newblock { Ann. Probab.}  \textbf{29}:1796--1832 (2001).

\bibitem{[Brz+5]}
Z. Brze\'zniak, B. Goldys, P. Imkeller, S.  Peszat, E. Priola and J.  Zabczyk,  \newblock \textit{  Time irregularity of generalized Ornstein-Uhlenbeck processes},  C. R. Math. Acad. Sci. Paris \textbf{348}, no. 5-6, 273–276 (2010)

\bibitem{[Cartan]} H. Cartan,  {\sc Calcul Differentiel}, Herman, Paris 1965


\bibitem{[Da prato]}G. Da Prato and J. Zabczyk, \textit{Stochastic equations in infinite dimensions}, Encyclopedia of Mathematics and its Applications, \textbf{44}, Cambridge University Press, Cambridge, (1992)

\bibitem{Dirksen_2014}  S. Dirksen, \textit{It\^o isomorphisms for $L^p$-valued Poisson stochastic integrals}. Ann. Probab. 42 , no. 6, 2595-2643 (2014)

\bibitem{Fernando+Sritharan_2010} P. W. Fernando, B. R{\"u}diger and S.S. Sritharan,
\textit{Mild solutions of stochastic Navier-Stokes equation with jump noise in $L^p$-spaces}, Math. Nachr.,  DOI: 10.1002/mana.201300248

\bibitem{[Fr+Weis]}
A.~Fr\"ohlich and L.~Weis,
 \textit{$H^\infty$ calculus and dilations},
Bull. Soc. Math. France \textbf{134}, no. 4, 487--508  (2006)
\bibitem{[Halmos]} P. R.  Halmos, \textit{Measure Theory}, D. Van Nostrand Company, Inc., New York, N. Y., 1950.
\bibitem{[Haus+Seidler]}
E.~Hausenblas and J.~Seidler,
\textit{A note on maximal inequality for stochastic convolutions},
Czechoslovak Math. J. \textbf{51}, no. 4, 785--790 (2001)
\bibitem{[Hausenblas]} E. Hausenblas, \textit{A note on the It\^{o} formula of stochastic integrals in Banach spaces}, Random Oper. Stochastic Equations  \textbf{14},  no. 1, 45--58, (2006)

\bibitem{[Ichikawa]} A. Ichikawa, \textit{Some inequalities for martingales and stochastic convolutions}, Stochastic Anal. Appl. \textbf{4}, no. 3, 329--339, (1986)

\bibitem{[Ikeda]}  N. Ikeda and S. Watanabe, Stochastic differential equations and diffusion processes. Second edition. North-Holland Mathematical Library, 24. North-Holland Publishing Co., Amsterdam; Kodansha, Ltd., Tokyo, 1989.

\bibitem{[Ito-0]}K. It\^{o},  \textit{Poisson point processes attached to Markov processes},
Proceedings of the Sixth Berkeley Symposium on Mathematical Statistics and Probability (Univ. California,
Berkeley, Calif., 1970/1971), Vol. III: Probability theory, pp. 225--239. Univ. California Press, Berkeley, California, 1972.


\bibitem{jacod}
J.Jacod, and A. Shiryaev.
{\em Limit theorems for stochastic processes.}
Grundlehren der Mathematischen Wissenschaften, 288. Springer-Verlag, Berlin, 1987.

\bibitem{[Kallenberg]} O. Kallenberg, Foundations of Modern Probability, Berlin, Springer, 2002.


\bibitem{[Kotelenez]} P. Kotelenez, \textit{A submartingale type inequality with applications to stochastic evolution equations},  Stochastics  \textbf{8}, no. 2, 139--151, (1982/83)


\bibitem{[Liu+Zhai_2012]} Y. Liu and J. Zhai,
\textit{A note on time regularity of generalized Ornstein-Uhlenbeck processes with cylindrical stable noise},
C. R. Math. Acad. Sci. Paris \textbf{350}, no. 1-2, 97–100  (2012)

\bibitem{marinelli}
C.~Marinelli and M.~R\"ockner,
{\sl Well-posedness and asymptotic behavior for stochastic reaction-diffusion equations with multiplicative Poisson noise.}
 Electron. J. Probab., \textbf{15}:1528-1555 (2010).


\bibitem{[Metivier]} M. M\'{e}tivier, \textit{Semimartingales, A course on stochastic processes}, de Gruyter Studies in Mathematics,
2. Walter de Gruyter Co., Berlin-New York, 1982.

\bibitem{[Neidhardt]} A.L. Neidhardt, \textit{Stochastic integrals in 2-uniformly smooth Banach spaces}, University of Wisconsin, 1978.

\bibitem{[Pazy]} A. Pazy, \textit{Semigroups of linear operators and applications to partial differential equations},
Applied Mathematical Sciences, \textbf{44}, Springer-Verlag, New
York, (1983)


\bibitem{Peszat_Z_2007}
S.~Peszat and J.~Zabczyk.
\newblock{\em Stochastic partial differential equations with L\'{e}vy noise.}
\newblock Encyclopedia of Mathematics and its Applications, 113. Cambridge University Press, Cambridge, 2007.

\bibitem{Pisier}G. Pisier, \textit{Martingales with values in uniformly convex spaces}. Israel J. Math 20 326-350 (1975)

\bibitem{Pisier86}G. Pisier, \textit{Probabilistic methods in the geometry of Banach space.} In: Probability and Analysis (Varenna). In: Lecture Notes in Springer, Berlin, pp. 167-241, 1986.

\bibitem{Revuz+Yor_1999}  D. Revuz and M.  Yor, {\sc Continuous martingales and Brownian motion.} Third edition. Grundlehren der Mathematischen Wissenschaften  \textbf{293}, Springer-Verlag, Berlin, 1999.

\bibitem{[Rong]}S. Rong, \newblock{\em Theory of Stochastic Differential Equations with Jumps and Applications,}\newblock Springer, New York, 2005

\bibitem{Rudiger_2004} B. R\"{u}diger, \textit{Stochastic integration with respect to compensated Poisson random measures on separable Banach spaces},  Stoch. Stoch. Rep.  \textbf{76},  no. 3, 213--242 (2004)




\bibitem{[Sato]} K. Sato, \newblock{\em L\'{e}vy processes and infinitely divisible distributions,}
\newblock Translated from the 1990 Japanese original. Revised by the author. Cambridge Studies in Advanced Mathematics, 68. Cambridge University Press, Cambridge, 1999.

\bibitem{seidler}
J.~Seidler,
\textit{Exponential estimates for stochastic convolutions in $2$--smmoth Banach spaces},  Electron. J. Probab. \textbf{15}, no. 50, 1556–1573 (2010)

\bibitem{[Tubaro]} L. Tubaro, \textit{An estimate of Burkholder type for stochastic processes defined by the stochastic integral}.  Stochastic Anal. Appl. \textbf{2},  no. 2, 187--192 (1984)


\bibitem{[Talay_1997]} P.~Protter and D.~Talay,  \textit{The {E}uler scheme for {L}\'evy driven stochastic differential
  equations},  {Ann. Probab.} \textbf{ 25}, 1, 393--423 (1997)

\bibitem{[Neerven+Zhu]}J. van Neerven, J. Zhu \textit{A maximal inequality for stochastic convolutions in 2-smooth Banach spaces}, Electron Commun Probab.1 \textbf{6}:689--705 (2011).


\bibitem{Wentzell} A. D. Wentzell, \textit{A course in the theory of stochastic processes}, Translated from the Russian by S. Chomet. With a foreword by K. L. Chung, McGraw-Hill International Book Co., New York, (1981).
\bibitem{Zhu_2010_PhD thesis} J. Zhu, A Study of SPDES w.r.t. compensated Poisson random measures and related topics, Ph.D Thesis, University of York, 2010.

\end{thebibliography}
\end{document}